\documentclass{article}
\usepackage{xcolor}
\definecolor{lilac}{rgb}{0.7,0.4,0.9}
\usepackage[utf8]{inputenc}
\usepackage{amsmath}
\usepackage{enumitem}
\usepackage{amsfonts}
\usepackage{setspace}
\usepackage{amssymb}
\usepackage{cancel}  
\usepackage{graphicx}
\usepackage{mathrsfs}
\usepackage{esint}
\usepackage{upref,amsthm,amsxtra,exscale}
\usepackage{cite}
\usepackage[colorlinks=true,urlcolor=blue,
citecolor=red,linkcolor=blue,linktocpage,pdfpagelabels,
bookmarksnumbered,bookmarksopen]{hyperref}
\usepackage{cleveref}
\usepackage[cm]{fullpage}

\usepackage{subcaption}
\usepackage{caption}

\numberwithin{equation}{section}

\newtheorem{theorem}{Theorem}[section]
\newtheorem{proposition}[theorem]{Proposition}
\newtheorem{lemma}[theorem]{Lemma}
\newtheorem{corollary}[theorem]{Corollary}
\newtheorem{example}[theorem]{Example}

\newtheorem{remark}{Remark}
\newtheorem{definition}{Definition}

\title{Infinitely many solutions to  Kirchhoff-Boussinesq-type equations on Riemannian manifolds }

\author{Romulo D. Carlos\footnote{ R.D. Carlos was supported by a postdoctoral fellowship from UNAM-DGAPA.},
Juan Carlos Fernández\footnote{ J.C. Fernández was partially supported by UNAM-DGAPA-PAPIIT IA103925}
\ and María de los Ángeles Sandoval-Romero}
	
\date{}
\begin{document}
    \maketitle

\begin{abstract}
We study Kirchhoff-Boussinesq-type equations on closed Riemannian manifolds $(M,g)$ of the form
\[
\sum_{j=0}^m a_i(-\Delta_g)^m u \pm \text{div}_g(\vert\nabla u\vert_g^{p-2}\nabla u) = f(u),\qquad \text{ on }\ M,
\]
where $m<\dim M/2$, $(a_0,a_1,\ldots a_m)\in C^\infty_+(M)\times [0,\infty)^{m-1}\times(0,\infty)$, $2<p\leq\frac{2\dim M}{\dim M- 2}$ and $f:\mathbb{R}\rightarrow\mathbb{R}$ is a continuous nonlinearity of superlinear type with subcritical or critical Sobolev growth. We briefly motivate the above equation in the case of $m=2$, as an extension of the stationary Kirchhoff-Boussinesq equation when modeling the dynamics of curved elastic plates. Under some symmetry assumptions, we prove the existence of multiple equivariant solutions when considering critical Sobolev nonlinearities and, for subcritical ones, we also prove the existence of ground-state solutions. As a byproduct, we prove a Gagliardo-Nirenberg interpolation inequality and give several equivalent norms on the higher order Sobolev space $H_g^m(M)$.
\medskip


\medskip


\end{abstract}

	\maketitle

	



    \section{Introduction}

The stationary Kirchhoff-Boussinesq equation on domains,  $\Omega\subset\mathbb{R}^d$ 
\begin{equation}\label{EuclidianK-B}
\Delta^2 u \pm \text{div}(\vert\nabla  u\vert^{p-2}\nabla u)  = f(u),\qquad\text{in }\Omega,
\end{equation}
and related problems, under different boundary conditions, have attracted considerable attention in recent years (see, for instance,\cite{carlos2024existence,tavares2025existence,carlos2024elliptic}).  This equation arises in the study of the dynamics of elastic plates accounting for transverse shear effects (see, for instance,\cite{ChueshovLasiecka2006,ChueshovLasiecka2006b,ChueshovLasiecka2011}). 

In this paper, we propose and study an extension of the Kirchhoff-Boussinesq equation to Riemannian manifolds. Concretely, if $(M,g)$ denotes a Riemannian manifold with dimension $d:=\dim M$ and if we consider real numbers $2<p<d$ and $r>p$, we study the existence of solutions to the problem
\begin{equation}\label{Problem:Kirchhof-Boussinesq}
\Delta_g^2 u \pm \Delta_{g,p} u + a_0 u = f_q(u) + \vert u\vert^{r-2}u \quad \text{on }\  M.
\end{equation}
where $a_0\in C^\infty(M)$, $f_q:\mathbb{R}\rightarrow\mathbb{R}$ is a nonlinear function satisfying suitable assumptions depending on $q\in\mathbb{R}$,  $\Delta_g$ is the Laplace-Beltrami operator on $(M,g)$ and $\Delta_{p,g}
:= \text{div}_g(\vert \nabla\cdot\vert^{p-2}\nabla\cdot)$ denotes the $p$-Laplacian operator on $(M,g)$.

The motivation to study equation \eqref{Problem:Kirchhof-Boussinesq} arises when modeling curved plates by means of a Riemannian manifold, where terms involving different types of curvature emerge naturally and are represented by the term $a_0$ in the equation. Indeed, when considering flat plates in the Mindlin-Timoshenko model, the Kirchhoff-Boussinesq equation appears when applying the divergence to the symmetrization of the derivative of a vector field in $\mathbb{R}^2$ (see \cite{ChueshovLasiecka2011} and \cite{LagneseBook,LagneseLionsBook}). This symmetrization is known as the \emph{deformation tensor} (see \cite{Czubak2024}) and, as remarked by Mazzucatto \cite[Section 5]{Mazzucatto2003}, the divergence of the deformation tensor coincides with the vectorial laplacian (Bochner Laplacian). However, in the case of Riemannian manifolds, the divergence of the stress tensor, known as the \emph{deformation laplacian} and introduced by Ebin and Marsden in \cite{EbinMarsden1970}, does not coincide with the Bochner laplacian (see \cite[Section 5]{Taylor1992}). Instead, if $\Delta_D$ denotes the deformation laplacian and $\Delta_B$ the Bochner laplacian, a Weitzenböch type formula (see, for example, \cite{Czubak2024,ChanCzubak2026}) yields for divergence free vector fields that
\[
-2\Delta_D=-\Delta_B-\text{Ric},
\]
where $\text{Ric}$ denotes the Ricci curvature. Observe that in the case of flat plates, the Ricci curvature is zero and the two operators coincide. Under the hypothesis of the transversal shear effects, these curvature terms should yield a smooth scalar function $a_0$ defined on $M$, as we proposed in equation \ref{Problem:Kirchhof-Boussinesq}.

To simplify our analysis, we will assume that the Riemannian manifolds are closed (compact and without boundary), with dimension $d>4$ and that the potential term \(a_{0}\) in equation \eqref{Problem:Kirchhof-Boussinesq} is strictly positive.  It is worth noting that potentials also arise in the Euclidean Kirchhoff-Boussinesq equation \eqref{EuclidianK-B} when considering $\Omega=\mathbb{R}^n$  (see, for instance, \cite{SunWu2017, carlos2024kirchhoff, carlos2025concentration, figueiredo2025result, carlos2024nonlinear}). 

Rather than restricting ourselves to the problem \eqref{Problem:Kirchhof-Boussinesq}, we will deal with a more general problem, which we now state. Let $m\geq 2$ be an integer, and let $(M,g)$ be a closed Riemannian manifold of dimension $d>2m$. Define the elliptic operator of order $2m$
\begin{equation}\label{Def:Operator}
P_{\textbf{a}}^m := \sum_{i=0}^m a_i(-\Delta_g)^i,\qquad a_1,\ldots a_{m-1}\in[0,\infty), a_m>0 \text{ and } a_0\in C^\infty(M,(0,\infty)), 
\end{equation}
where, as usual $(-\Delta_g)^0:= Id$. For $p\in(2,d)$, we consider the following  elliptic problem involving a combination of the $2m$-th order $P_{\textbf{a}}^m$ and a $p$-Laplacian,  
    \begin{equation}\label{Problem:Main}
	 P_{\textbf{a}}^mu \pm (\Delta_{g})_{p} u =  f_q(u)+ |u|^{r-2}u \qquad \text{ on } M,
\end{equation}
where  $2<p< r \leq 2_{m,d}^{\ast}$, $2_{m,d}^{\ast}:=\frac{2d}{d-2m}$ being the critical Sobolev exponent of order $m$, 
and $f_q:\mathbb{R}\rightarrow\mathbb{R}$ is a continuous nonlinearity  satisfying the following hypotheses 
\begin{enumerate}[label=($f_1$),ref=$(f_1)$]
		\item \label{f_1} \mbox{There exist constants $C>0$ and  $q\in(p,r)$ such that the function  $f$ satisfies $  |f_q(t)|\leq C(1+|t|^{q-1}).$} 
\end{enumerate}

\begin{enumerate}[label=($f_2$),ref=$(f_2)$]
		\item \label{f_2} \mbox{The limit holds $
 \lim_{|t| \rightarrow 0} \frac{f_q(t)}{|t|} =0
$ }  holds.
  \end{enumerate}

  \begin{enumerate}[label=($f_3$),ref=$(f_3)$]
		\item \label{f_3} \mbox{The function $t \rightarrow \frac{f_q(t)}{|t|^{p-2}t}$ is increasing in  $(0, +\infty)$. } 
  \end{enumerate}

 \begin{enumerate}[label=($f_4$),ref=$(f_4)$]
		\item \label{f_4} $f_q(-t)=-f_q(t)$, for all $t \in \mathbb{R}$.
		
  \end{enumerate}
Observe that we recover the Kirchhoff-Boussinesq equation \eqref{Problem:Kirchhof-Boussinesq} by taking $m=2$ and $a_1=0$. Also notice that we are allowing critical Sobolev nonlinearities in \eqref{Problem:Main} when $r=2_{m,d}^*$, which lack the compactness needed in order to apply standard variational methods (see, for instance, \cite{StruweBook,WillemBook}).  To deal with the critical exponent nonlinearity, we introduce symmetries as in \cite{FernandezPalmasTorres2024}. 

 In order to state our main results, we briefly review the terminology regarding group actions on manifolds. If $\text{Isom}(M,g)$ denotes the group of isometries of $(M,g)$, there is an action by isometries on $M$  defined by $(\gamma, x)\mapsto \gamma x:=\gamma(x), x\in M,\gamma\in \text{Isom}(M,g)$. Given a closed subgroup $\Gamma\subset \text{Isom}(M,g)$, we study the existence and multiplicity of solutions to problem \eqref{Problem:Main} that are \emph{$\Gamma$-invariant,} that is, satisfying  $u(\gamma x)=u(x)$ for every $x\in M$ and every $\gamma\in \Gamma$.  Given $x\in M$, define the $\Gamma$-orbit of $x$ as the set
\[
\Gamma x:= \{ \gamma x\in M \;:\; \gamma\in\Gamma \}.
\]
Since $\Gamma$ is closed, the $\Gamma$-orbits are submanifolds of $M$ (see, for instance, \cite{AlexBettiol}). To ensure that some Sobolev spaces with symmetries are infinite-dimensional (see Section \ref{Section:SymmetricSetting} below), we will impose the following condition on $\Gamma$:
\begin{equation}\label{ActionNotTransitive}
\dim \Gamma x < d.
\end{equation}
It is worth noting that the non symmetric scenario is obtained when $\Gamma=\{Id:M\to M\}$.

Recall that a nontrivial solution having the least energy associated with equation \eqref{Problem:Main} is called a ground state solution (see Section \ref{Section:VariationalSetting} below for a precise definition).
We are ready to state our main result, which extends  the main results in\cite{CarlosFurtado, carlos2024existence, tavares2025existence, santaria2026class,figueiredo2026result} to higher order operators and on manifolds, allowing critical Sobolev exponent nonlinearities.

\begin{theorem}\label{Theorem:Main}
Let $m\geq 2$ be an integer, $(M,g)$ be a closed Riemannian manifold of dimension $d>2m$, $\Gamma$ be a closed subgroup  of $\emph{Isom}(M,g)$ satisfying \eqref{ActionNotTransitive}, and $p,q,r\in\mathbb{R}$ be such that
\begin{equation}\label{Hypothesis:pq,r}
2<p<d\qquad\text{and}\qquad p<q<r\leq 2_{m,d}^*.
\end{equation}
Assume that $f_q:\mathbb{R}\rightarrow\mathbb{R}$ is continuous and satisfies \ref{f_1} to \ref{f_4} and that the coefficients of the operator $P_{\textbf{a}}^m$ satisfy that 
\begin{equation}\label{Hypothesis:Coefficients}
a_m>0, \quad a_1,\ldots,a_{m-1}\in[0,\infty) \quad\text{ and } \quad a_0\in C^\infty(M) \text{ is positive and }\Gamma\text{-invariant}.
\end{equation} 
If either
\begin{equation}\label{Equation:p_subcritical}
m=2 \text{ and } 2<p< 2_{1,d}^\ast\qquad\text{or}\qquad  m\geq 3 \text{ and } 2<p\leq 2_{1,d}^\ast 
\end{equation}
hold and if either
\begin{enumerate}[label=(\roman*)]
\item $r<2_{m,d}^*$ 

\hspace{-1cm} or

\item $r=2_{m,d}^*$ and
\begin{equation}\label{Hyp:Gamma}
1\leq \dim \Gamma x\leq m-1,\quad\text{for every }x\in M,
\end{equation}
\end{enumerate}
then, the problem 
\eqref{Problem:Main} admits an infinite number of $\Gamma$-invariant weak solutions with increasing energies, and one of them has the least energy among all other $\Gamma$-invariant solutions. 

In particular, when $r<2_{m,d}^*$ and $\Gamma=\{Id:M\to M\}$, there exists a ground state solution to \eqref{Problem:Main}.
 \end{theorem}

\medskip

As an immediate consequence, we have a multiplicity result for the Kirchhoff-Boussinesq equation on a closed Riemannian manifold.

\begin{corollary}
Let $(M,g)$ be a closed Riemannian manifold of dimension $d\geq 5$, $\Gamma$ a closed subgroup of $\emph{Isom}(M,g)$ satisfying \eqref{ActionNotTransitive} and $p,q,r\in\mathbb{R}$ satisfying \eqref{Hypothesis:pq,r}.
Assume that $f_q:\mathbb{R}\rightarrow\mathbb{R}$ is continuous and satisfies  \ref{f_1} to \ref{f_4}, and that $a_0\in C^\infty(M)$ is positive and $\Gamma$-invariant. 
If $p$ satisfies \eqref{Equation:p_subcritical} and if either
\begin{enumerate}[label=(\roman*)]
\item $r<2_{m,d}^*$ 

\hspace{-1cm} or

\item $r=2_{m,d}^*$ and $\Gamma$ satisfies \eqref{Hyp:Gamma},
\end{enumerate}
 then the nonlinear Kirchhoff-Boussinesq equation \eqref{Problem:Kirchhof-Boussinesq} admits an infinite number of $\Gamma$-invariant solutions with increasing energies, and one of them has least energy among any other $\Gamma$-invariant solutions. 

In particular, when $r<2_{m,d}^*$ and $\Gamma=\{Id\}$, there exists a ground state solution to \eqref{Problem:Kirchhof-Boussinesq}.
\end{corollary}

 We remark that our results
remain valid for a wide class of superlinear nonlinearities with
subcritical growth. The following examples illustrate the admissible
class of nonlinearities satisfying assumptions \ref{f_1}-\ref{f_4}.
\begin{example}
Let \(p>2\). The following nonlinearities satisfy assumptions
\ref{f_1}-\ref{f_4}:
\begin{enumerate}
    \item $f_q(t)= |t|^{q-2}t,$ with $p<q<2^*_{m,d}$.

     \item $f_q(t)= |t|^{q-2}t \,\ln (1+|t|^{\kappa})$ with $p<q$, $0<\kappa$ and $q+\kappa<2^*_{m,d}$.

    \item $f_{q_1,q_2}(t)= |t|^{q_1-2}t + |t|^{q_2-2}t,$ with $p<q_1<q_2<2^*_{m,d}$.

 \item $f_q(t)= \dfrac{|t|^{q-2}t}{1+|t|^{\alpha}},$ 
    with $0<\alpha<q-p$ and $p<q<2^*_{m,d}$.

    \item $f_q(t)= |t|^{q-2}t \,(1-e^{-|t|}),$ with $p<q<2^*_{m,d} $.
\end{enumerate}
\end{example}

\medskip

When $(M,g)$ is an Einstein manifold with positive scalar curvature, as it was shown by Gover in \cite{Gover06}, the higher order conformal operators introduced by Graham, Jenne, Mason and Sparling (GJMS for short) \cite{GJMS} have the form \eqref{Def:Operator} with $a_0>0$ constant. As the second order conformal operator in $\mathbb{R}^d$ is just the bilaplacian $(-\Delta)^2$, equation \eqref{Problem:Main} is a generalization of Kirchhoff-Boussinesq equation to these kinds of operators. In light of these considerations, the main theorem yields the following consequence.

\begin{corollary} \label{Corollary:Kirschoff-BoussinesqJGMS}
Let $(M,g)$ be a closed Einstein manifold of dimension $d\geq 5$ and positive scalar curvature,  and consider the $GJMS$-operator $\mathscr{P}_g^m$ of order $2\leq m<d/2$ related to $(M,g)$. Consider $p,q,r\in\mathbb{R}$ satisfying \eqref{Hypothesis:pq,r} and \eqref{Equation:p_subcritical}, and assume that $f_q:\mathbb{R}\to\mathbb{R}$ is continuous and satisfies  \ref{f_1} to \ref{f_4}. Then, for any closed subgroup $\Gamma$ of $\text{Isom}(M,g)$  satisfying \eqref{ActionNotTransitive} and if $r>q$ satisfies either  
\begin{enumerate}[label=(\roman*)]
\item $r<2_{m,d}^*$

\hspace{-1cm} or

\item $r=2_{m,d}^*$ and $\Gamma$ satisfies \eqref{Hyp:Gamma},
\end{enumerate}
then the following problem 
\begin{equation}\label{Problem:EinsteinManifolds}
	 \mathscr{P}_{g}^mu \pm (\Delta_{g})_{p} u =  f_q(u)+ |u|^{r-2}u \qquad \text{ in } \  M,
\end{equation}
admits an infinite number of $\Gamma$-invariant solutions with increasing energies, and one of them has least energy among any other $\Gamma$-invariant solutions.

In particular, when $r<2_{m,d}^*$ and $\Gamma=\{Id\}$, there exists a ground state solution to \eqref{Problem:Kirchhof-Boussinesq}.
\end{corollary}

It is interesting to compare this result with the main result in the recent article \cite{OunaneTahri2026}, where the authors considered $r\in(1,2)$ for these kind of equations, in the absence of the $p$-laplacian term. However, even if the operators $\mathscr{P}_g^m$ are important in Conformal Geometry, we do not know whether equation \eqref{Problem:EinsteinManifolds} has geometric significance or not when considering the critical Sobolev exponent nonlinearity $r=2_{m,d}^*$. It may be interesting to know if there is a conformal operator related to the Kirchhoff-Boussinesq operator in the critical case. 

\bigskip 

In order to prove Theorem \ref{Theorem:Main} and its corollaries, we will follow the method of the Nehari set given in \cite{SzulkinWeth2010}. An important part of the proof is to handle the lower order terms of the operator $P_{\textbf{a}}^m$ and the $p$-Laplacian. This control is obtained by applying a Gagliardo-Nirenberg interpolation inequality on manifolds to the $j$-th covariant derivative $\nabla^ju$ of $u\in C^\infty(M)$, which we next state. 

\begin{theorem}\label{Theorem:Gagliardo-Nirenberg}
Let $(M,g)$ be a closed Riemannian manifold of dimension $d$. Let   $j, m$ be non-negative integers satisfying $0< j<m$ and $2m<d$,  $1\leq \kappa_1, \kappa_2, \kappa_3< \infty$ be real numbers and let $\Upsilon \in [0,1]$  such that
\begin{align*}
    \frac{1}{\kappa_1}=\frac{j}{d}+ \left(\frac{1}{\kappa_3}-\frac{m}{d}\right)\Upsilon + \frac{(1-\Upsilon)}{\kappa_2} \quad \mbox{and} \quad \frac{j}{m}\leq \Upsilon \leq 1.
\end{align*}
Then, there exist a constant $C=C(d,m,j,\kappa_1,\kappa_2,\kappa_3,\Upsilon,M)> 0$ such that
\begin{align*}
  \Big( \int_{M}  |\nabla^{j}_{g} u |^{\kappa_1} \,dV_g    \Big)^{\frac{1}{\kappa_1}}  \leq C \Big[   \Big(\int_{M}  |\nabla^{m}_{g} u |^{\kappa_3} \, dV_g     \Big)^{\frac{1}{\kappa_3}}+  \Big( \int_{M}  |u |^{\kappa_3}  \, dV_g  \Big)^{\frac{1}{\kappa_3}}    \Big]^{\Upsilon} \Big( \int_{M}  |u |^{\kappa_2} \, dV_g  \Big)^{\frac{1-\Upsilon }{\kappa_2}},
\end{align*}
for any $u\in H_g^m(M).$
\end{theorem}
Here, as usual, $H_g^m(M)$ denotes the Sobolev space of order $m$ on $(M,g)$ (see Section \ref{Section:VariationalSetting} below). We want to remark that, even if this inequality is well known in the case of the Euclidean space (see \cite{Nirenber1959}), to the best of our knowledge, it was not known for Riemannian manifolds in the way tha we state it. For instance, this inequality was recently proved in \cite{DellaCorteDianaMantegazza2024} for manifolds which are graph hypersurfaces in $\mathbb{R}^d$, and by Aubin in the classical text \cite[Theorem 3.70]{AubinBook}, for functions with zero mean. In Appendix \ref{Appendix:ProofGagliardo-Nirenberg} we prove Theorem \ref{Theorem:Gagliardo-Nirenberg} by extending the proof given by Aubin to any smooth function $u:M\to\mathbb{R}$. As another consequence of Theorem \ref{Theorem:Gagliardo-Nirenberg}, in Appendix \ref{Subsection:Appendix:ProofEquivalentNorms} we prove that several interesting and familiar norms in $H_g^m(M)$ are equivalent. Even if this is well known, it is not easy to find the proof of these equivalences in the existing literature. Therefore, we present a unified approach based on Theorem \ref{Theorem:Gagliardo-Nirenberg}, integration by parts, and the G\aa rding inequality that might be interesting in its own right.

\medskip

\begin{remark}
It is worth noticing that the exponent \(p = 2^*_{1,d}\) appears as a limiting value in the corresponding Gagliardo–Ni\-ren\-berg interpolation inequality, but its criticality depends entirely on whether \(m = 2\) (the case of the pure Kirchhoff–Boussinesq equation). In the case of \(m > 2\) and \(p = 2^*_{1,d}\), the constant \(\Upsilon \) in the interpolation inequality is strictly less than one, allowing us to control the terms arising from the \(p\)-Laplacian in terms of the Sobolev norm and the \(L^{2}\)-norm, yielding the Palais–Smale condition (see Propositions \ref{Proposition:CriticalGN} and \ref{Propo:Palais-Smale} below). However, when \(m = 2\) and \(p = 2^*_{1,d}\), the constant becomes \(\Upsilon = 1\), and the \(L^{2}\)-norm term in Propositions \ref{Proposition:CriticalGN} vanishes, preventing us from proving the Palais–Smale condition in this setting. Moreover, even when symmetries allow us to handle the critical nonlinearity \(r = 2^*_{m,d}\), this approach does not seem to work for the critical parameter \(p\), which calls for new ideas.
Consequently, it remains unclear whether this exponent represents a genuine threshold for compactness in the Kirchhoff–Boussinesq setting. In particular, it would be interesting to determine whether the embedding associated with the Kirchhoff–Boussinesq operator loses compactness at \(p = 2^*_{1,d}\) and, if so, whether one can construct an explicit concentrating sequence exhibiting such a loss of compactness. This leads to the following open question: Does the limiting exponent \(p = 2^*_{1,d}\) constitute a genuine critical threshold for compactness? A positive answer would provide a concrete concentration–compactness framework for the analysis of Kirchhoff–Boussinesq problems at this limiting exponent.
\end{remark}

\bigskip

This paper is organized as follows. In Section \ref{Section:VariationalSetting} we provide the functional framework related with the Nehari set method described in \cite{SzulkinWeth2010} and study the main properties of this set. Section \ref{Section:SymmetricSetting} is devoted to the variational setting with symmetries needed to handle the lack of compactness of the critical Sobolev exponent nonlinearities. Finally, in Section \ref{Section:CriticalPointTheory}, we sketch the critical point theory given in \cite{SzulkinWeth2010}, and we prove our main result and its corollaries. We also include an Appendix where we give the proof of the Gagliardo-Nirenberg interpolation inequality, Theorem \ref{Theorem:Gagliardo-Nirenberg}, and as a consequence of this and G\aa rding's inequality, we provide several equivalent norms in the space $H_g^m(M)$, which may be of independent interest.


    
	\section{Variational setting and Nehari set}\label{Section:VariationalSetting}

\medskip
Let $m\geq 2$ be an integer, $(M,g)$ be a closed Riemannian manifold of dimension $d>2m$, and $p,q,r\in\mathbb{R}$ satisfying $2<p<d$ and $p<q<r\leq 2_{m,d}^*$. 

In what follows, we denote the Lebesgue spaces associated with $(M,g)$ by $L_g^s(M), s\geq 1$, with the norm 
\[
\Vert u\Vert_s:=\left( \int_M \vert u\vert^s \right)^{1/s}.
\]
We define the Sobolev space $H_g^m(M)$ as the completion of $C^\infty(M)$ with respect to the norm
\begin{equation}\label{Equation:UsualSobolevNorm}
\|v \|_{H_g^{m}(M)}=\left[ \sum_{j=0}^{m}  \int_{M}   |\nabla^{j} v |^{2} dV_g  \right]^{1/2}, 
\end{equation}
endowed with the inner product
\begin{equation}\label{Equation:UsualSobolevInnerProduct}
\langle u,v \rangle_{H_g^{m}(M)}:= \sum_{j=0}^{m}  \int_{M}   \langle\nabla^{j} u , \nabla^{j}v\rangle_g\; dV_g,\quad u,v\in H_g^m(M) , 
\end{equation}
where $\nabla^j u$ denotes the $j$-th covariant derivative of $u$ and 
\[
\langle \nabla^j u,\nabla^j v\rangle_g:=g^{i_1k_1}\cdots g^{i_jk_j}(\nabla^j u)_{i_1\cdots i_j}(\nabla^j v)_{k_1\cdots k_j},
\]
denotes the Riemannian metric extended to tensor fields.

Denote by $C^\infty_+(M)$ the subset of positive smooth functions defined on $M$. For a fixed $\textbf{a}:=(a_0,a_1,\ldots,a_{m-1},a_m)\in C^\infty_+(M)\times[0,\infty)^{m-1}\times(0,\infty)$, recall the operator $P_{\textbf{a}}^m$ given in \eqref{Def:Operator}.  Green's identities imply that
\begin{equation}\label{Equation:PowerLaplacian}
\int_M v(-\Delta_g)^{i}u \; dV_g =
\begin{cases}
\int_M \Delta_g^{i/2} v\Delta_g^{i/2} u \; dV_g, & i \text{ even,} \\
\int_M \langle \nabla\Delta_g^{(i-1)/2}v, \nabla\Delta_g^{(i-1)/2} u\rangle_g \; dV_g, & i \text{ odd},
\end{cases}
\end{equation}
for any $u,v\in C^\infty(M)$. It follows that 
\begin{equation} \label{Equation:InteriorProductOperator}
\begin{split}
\langle u,v \rangle_{\textbf{a}}&:=\int_M vP_{\textbf{a}}^mu\ dV_g\\
&= \sum_{\substack{i=0\\ i\  even}}^m  \int_M a_i \Delta_g^{i/2} v\Delta_g^{i/2} u \; dV_g
+ \sum_{\substack{i=0\\ i\  odd}}^m \int_M a_i\langle \nabla\Delta_g^{(i-1)/2}v, \nabla\Delta_g^{(i-1)/2} u\rangle_g \; dV_g,\quad   \forall\;u,v\in C^\infty(M).
\end{split}
\end{equation}
defines a symmetric bilinear form that extends  to a continuous symmetric bilinear form in $H_g^m(M)\times H_g^m(M)$, see \cite{Robert2011}. Since $a_0>0$, this bilinear form is an inner product in $H_g^m(M)$ and induces a norm in this space given by
\begin{equation}\label{Equation:EquivalentNorm}
\Vert u\Vert_{\textbf{a}}:=\sqrt{\langle u,u \rangle_{\textbf{a}}} = \left( \sum_{\substack{i=0\\ i\  even}}^m  \int_M a_i \vert \Delta_g^{i/2}  u\vert^2 \; dV_g
+ \sum_{\substack{i=0\\ i\  odd}}^m \int_M a_i \vert \nabla\Delta_g^{(i-1)/2}u\vert_g^2 \; dV_g\right)^{1/2},\quad u\in H_g^m(M)..
\end{equation}
Notice that
\[
\Vert u\Vert_{\textbf{a}}=\sqrt{\int_M u P_{\textbf{a}}^m u\; dV_g}\qquad\forall\;u\in C^\infty(M).
\]
Moreover, as a consequence of the Gagliardo-Nirenberg interpolation formula given in Theorem \ref{Theorem:Gagliardo-Nirenberg}, integration by parts, and G\aa rding's inequality, we show in Appendix \ref{Subsection:Appendix:ProofEquivalentNorms} that  $\Vert\cdot\Vert_{\textbf{a}}$ defines an equivalent norm in $H_g^m(M)$. We summarize this fact in the following result.

\begin{proposition}\label{Proposition:EquivalentNorms}
For any $\textbf{a}:=(a_0,a_1,\ldots,a_{m-1},a_m)\in C^\infty_+(M)\times[0,\infty)^{m-1}\times(0,\infty)$, the norm $\Vert \cdot\Vert_{\textbf{a}}$ is equivalent to the usual Sobolev norm $\Vert\cdot\Vert_{H_g^m(M)}$.
\end{proposition}

We also require the following consequence of the Gagliardo-Nirenberg interpolation inequality, relating the $L^p$-norm of the gradient with the  norm $\Vert\cdot\Vert_{\textbf{a}}$.

\begin{proposition}\label{Proposition:CriticalGN}
For any $m,d\in\mathbb{N}$ such that $2m<d$ and any $2\leq p\leq 2_{1,d}^*=\frac{2d}{d-2}$, there exists $\Upsilon\in[\frac{1}{m},\frac{2}{m}]$ and a positive constant $C>0$ such that
\[
\left( \int_M \vert \nabla u\vert^p_g \; dV_g\right)^{1/p}\leq C \Vert u\Vert_{\textbf{a}}^{\Upsilon}\Vert u\Vert_2^{(1-\Upsilon)},\qquad\forall\; u\in C^\infty(M).
\]
Moreover
\[
\Upsilon = \begin{cases} \frac{1}{m} & \text{iff}\quad p=2\\
 \frac{2}{m} & \text{iff}\quad p=2_{1,d}^*=\frac{2d}{d-2}.
\end{cases}
\]
In particular, when $m=2$ and $p=2_{1,d}^*$, $\Upsilon=1$ and the inequality becomes 
\[
\left( \int_M \vert \nabla u\vert^p_g \; dV_g\right)^{1/p}\leq C \Vert u\Vert_{\textbf{a}}.
\]
\end{proposition}

\begin{proof}
For $k_1=p\in[2,2^*_{1,d}], \kappa_2=\kappa_3=2$, $j=1$, set $\Upsilon\in\mathbb{R}$ given by the identity
\begin{equation}\label{Equation:UpsilonP-Laplacian}
\frac{1}{p} = \frac{1}{d} + \frac{1}{2} - \frac{\Upsilon m}{d}. 
\end{equation}
We now confirm that this choice of constants satisfies the hypotheses of Theorem \ref{Theorem:Gagliardo-Nirenberg}. Indeed, we have that
\[
\frac{1}{\kappa_1}=\frac{1}{p} = \frac{1}{d} + \frac{1}{2} - \frac{\Upsilon m}{d} = \frac{1}{d} + \left( \frac{1}{2} - \frac{m}{d} \right)\Upsilon + \frac{1-\Upsilon}{2} 
= \frac{j}{d} + \left( \frac{1}{\kappa_3} - \frac{m}{d} \right)\Upsilon + \frac{1-\Upsilon}{\kappa_2},
\]
 Thus, $\Upsilon$ satisfies the first identity in Theorem \ref{Theorem:Gagliardo-Nirenberg}; we verify that $\Upsilon\in [1/m,2/m]=[j/m,1]$. To see this, as $m\geq 2$, then $d>4$ and
\[
2\leq p \leq 2^*_{1,d}=\frac{2d}{d-2} < d,
\]
thus, by \eqref{Equation:UpsilonP-Laplacian}, we get that
\[
\frac{d-2}{2d}\leq \frac{1}{p} = \frac{1}{d} + \frac{1}{2} - \frac{\Upsilon m}{d} \leq \frac{1}{2}.
\]
From the left-hand side of this inequality, we get that
\[
\Upsilon\leq \frac{2}{m}\leq 1
\]
because $m\geq 2$. Moreover, $\Upsilon=\frac{2}{m}$ if and only if $p=2_{1,d}^*$, and in this case, $\Upsilon=1$ if and only if $m=2$. 

Now, from the right-hand side of the same inequality, we obtain that 
\[
\frac{j}{m}=\frac{1}{m}\leq\Upsilon,
\]
with equality if and only if $p=2$.
Therefore, $\Upsilon\in[1/m,1]$ and by Theorem \ref{Theorem:Gagliardo-Nirenberg}, there exists $C_p>0$ such that
\begin{align*}
&\left( \int_M \vert \nabla u\vert_g^p \;dV_g \right)^{1/p}
\leq C_p\left[   \left(\int_M \vert \nabla^m u\vert_g^2  \;dV_g\right)^{1/2} + \left(\int_M u^2  \;dV_g\right)^{1/2} \right]^\Upsilon
\left[\int_M u^2\; dV_g \right]^{(1-\Upsilon)/2}.
\end{align*}
Then, by Proposition \ref{Proposition:EquivalentNorms}, we have that there exists a constant $C_\textbf{a}>0$ such that
\begin{align*}
\left( \int_M \vert \nabla u\vert_g^p \;dV_g \right)^{2/p}
&\leq C_p^2 \left[   \left(\int_M \vert \nabla^m u\vert_g^2  \;dV_g\right)^{1/2} + \left(\int_M u^2  \;dV_g\right)^{1/2} \right]^{2\Upsilon}
\left[\int_M u^2\; dV_g \right]^{(1-\Upsilon)}\\
&\leq 4C_p^2 \left[  \int_M \vert \nabla^m u\vert_g^2  \;dV_g + \int_M u^2  \;dV_g\right]^{\Upsilon}
\left[\int_M u^2\; dV_g \right]^{(1-\Upsilon)}\\
&\leq 4C_p^2\left[ \sum_{j=0}^m \int_M \vert \nabla^j u\vert_g^2  \;dV_g \right]^{\Upsilon}\left[\int_M u^2\; dV_g \right]^{(1-\Upsilon)}\\
&\leq 4C_{\textbf{a}}C_p^2\Vert u\Vert_{\textbf{a}}^{2\Upsilon}\Vert u\Vert_2^{2(1-\Upsilon)},
\end{align*}
where we conclude by taking $C:=2C_{\textbf{a}}^{1/2}C_p$.
\end{proof}

By Proposition \ref{Proposition:CriticalGN}, for any $u\in H_g^m(M)$ we have that $\vert \nabla u\vert_g\in L^p_g(M)$ and by the Cauchy-Schwarz inequality applied to the metric $\langle \cdot,\cdot\rangle_g$ together with Hölder inequality, for any $u,v\in H_g^m(M)$, we have that $\vert\nabla u\vert^{p-2}\langle \nabla u,\nabla v\rangle_g\in L^1(M)$. Also, by the Sobolev embedding 
\begin{equation}\label{Equation:SobolevEmbedding}
H_g^{m}(M)\hookrightarrow L_g^s(M),\quad s\in[1,2_{m,d}^*],
\end{equation}
together with the compactness of $M$, the growth estimate \ref{f_1} and the Hölder inequality, we have that $f_q(u)uv,\ \vert u\vert^{r-2}uv\in L_g^1(M)$. Consequently, we are now in a position to define the weak solutions to problem \eqref{Problem:Main}.

\begin{definition}
   A function $u\in H^{m}_{g}(M)$ is called a weak solution of \eqref{Problem:Main} if 
\begin{align*}
   \langle u,v \rangle_{\textbf{a}} \pm \int_{M}  |\nabla u |^{p-2} \langle \nabla u, \nabla v \rangle_g \, dV_g= \int_{M} f_q(u) v \,  dV_g  +\int_{M}  | u|^{r-2}  u  v \,  dV_g , 
\end{align*}
for all $v \in H^{m}_{g}(M)$.
\end{definition}

Denote by $F_q$ the antiderivative of $f_q$, that is 
\[      F_q(t)=\begin{cases}
        \int_0^t f_q(s) \; ds & t\geq0\\
        -\int_t^0 f_q(s)\; ds & t\leq0
        \end{cases},
        \]
and define the energy functional $I:H^{m}_{g}(M) \to \mathbb{R}$ 
\begin{align*}
   I(u) := \dfrac{1}{2} \| u \|_\textbf{a}^{2}  \pm \frac{1}{p} \int_{M}  |\nabla u |^{p} \,  dV_g - \int_{M} F_q(u) \, dV_g   -\frac{1}{r} \int_{M}  |u |^{r}  \, dV_g  
\end{align*}
Based on the previous remarks, this functional is well defined. Since $2<p<q<r\leq 2_{m,d}^*$, and as $f_q$ is continuous and satisfies \ref{f_1} and \ref{f_2},  Lemma \ref{Proposition:CriticalGN} and the Sobolev embedding \eqref{Equation:SobolevEmbedding} yield that $I$ is a $C^1$-functional, and its derivative is given by

\begin{equation}\label{Equation:DerivativeFunctional}
  I^{\prime}(u)v   =\langle u,v \rangle_{\textbf{a}} \pm \int_{M}  |\nabla u |^{p-2} \langle \nabla u,  \nabla v \rangle_g\,  dV_g - \int_{M} f_q(u) v \,  dV_g  -\int_{M}  |u |^{r-2}  u  v  \,  dV_g   
\end{equation}
for all $u,v \in H^{m}_{g}(M)$. Hence, weak solutions to \eqref{Problem:Main} are precisely the critical points of $I$. The nontrivial ones lie in the Nehari set associated to the functional $I$, given by
 \begin{equation}\label{DefinitionNehari}
 \begin{split}
     \mathcal{N}&:= \{u \in  H^{m}_{g}(M) \backslash \{0\}:\ I'(u)u =0 \}\\
     &=\left\{u \in  H^{m}_{g}(M) \backslash \{0\}\ :\ \Vert u\Vert_{\textbf{a}}^2 \pm \int_{M}  |\nabla u |^{p} \,  dV_g = \int_{M} f_q(u) u \,  dV_g +\int_{M}  |u |^{r}   \,  dV_g   \right\}
     \end{split}
 \end{equation}

 In Lemma \ref{Lemma:NehariNonEmpty} below, we show that $\mathcal{N}$ is nonempty. Notice that
 \begin{equation}\label{Equation:FunctionalDefinedOnNehari}
 I(u)= \left(\frac{p-2}{2p}\right)\Vert u\Vert_{\textbf{a}}^2 + \frac{1}{p}\int_M \left[ f_q(u)u - pF_q(u)  \right]\;dV_g + \left(\frac{r-p}{rp}\right)\int_M \vert u\vert^r\; dV_g,\quad\forall\;u\in\mathcal{N}.
 \end{equation}

 To see that $\mathcal{N}$ is nonempty, we need the following technical result.

 \begin{lemma}\label{Lemma:Properties-f} Let $q\in(p,2_{m,d}^*)$.
 \begin{enumerate}[label=(\roman*)]
 \item \label{Item:Crecimiento} If conditions \ref{f_1} and \ref{f_2} hold true, then, for any $\varepsilon>0$, there exists $C_\varepsilon>0$ such that
\begin{align}\label{crescimentoforigem}
|f_q(t)|\leq \varepsilon |t|+ C_\varepsilon |t|^{q-1}.  
\end{align}
and
\begin{align}\label{Crescimentoforigem}
|F_q(t)|\leq \frac{\varepsilon}{2} |t|^2+ \frac{C_\varepsilon}{q} |t|^{q}, 
\end{align}
for every $t\in\mathbb{R}$.
\item \label{Item:CrecimientoCombinado} If conditions \ref{f_3} and \ref{f_4} hold true, then the map $\xi:\mathbb{R}\rightarrow\mathbb{R}$ given by
\begin{equation}\label{crescente}
   \xi(t):= tf_q(t)-pF_q(t)\ 
\end{equation}
is increasing for $t \in (0,\infty)$ and decreasing for $t \in (-\infty,0)$. In particular, 
\[
\xi(t)=tf_q(t)-pF_q(t)> 0, \qquad \forall\;t\in \mathbb{R} \setminus \{0\}.
\]
 \end{enumerate}
 
 \end{lemma}
 
 \begin{proof}
 By condition \ref{f_2},   given $\varepsilon>0$, there exists $\delta=\delta(\varepsilon)>0$ such that $|t| < \delta$ implies $ |f_q(t)| < \varepsilon |t|$.

For $|t|\ge \delta$, we use the inequality
\begin{equation*}
    1 = \frac{1}{|t|^{q-1}} |t|^{q-1} \leq \frac{1}{\delta^{q-1}} |t|^{q-1},
\end{equation*}
together  with \ref{f_1} to conclude that
\begin{equation*}
    |f_q(t)| \le C (1 + |t|^{q-1}) \le C(1 + \delta^{-(q-1)}) |t|^{q-1},\qquad \vert t\vert\geq\delta,
\end{equation*}
So, if we define $C_\varepsilon := C (1 + \delta^{-(q-1)})$, which depends on $\varepsilon$ because $\delta=\delta(\varepsilon)$, we have that
\[
|f_q(t)| \le \varepsilon |t| \le \varepsilon |t| + C_\varepsilon |t|^{q-1}, \quad \vert t\vert<\delta
\] 
and 
\[
|f_q(t)| \le C_\varepsilon |t|^{q-1} \le \varepsilon |t| + C_\varepsilon |t|^{q-1}, \quad \vert t\vert\geq \delta,
\]
which proves the first estimate of item \ref{Item:Crecimiento}.

 Furthermore, from the definition of $F_q(t)$, it follows that
\begin{equation*}
    |F_q(t)| \le \int_0^{|t|} (\varepsilon s + C_\varepsilon s^{q-1}) ds = \frac{\varepsilon}{2} |t|^2 + \frac{C_\varepsilon}{q} |t|^q.
\end{equation*}
This proves \ref{Item:Crecimiento}.

\medskip

To prove the second assertion, suppose $0<s<t$. First,  
 \ref{f_3} yields that
\[
s f_q(s) = \frac{f_q(s)}{s^{p-1}} s^p < \frac{f_q(t)}{t^{p-1}} s^p.
\]
In a similar way,
\begin{equation*}
    p \int_s^t f_q(\tau) \, d\tau = p \int_s^t \frac{f_q(\tau)}{\tau^{p-1}} \tau^{p-1} \, d\tau
< p \int_s^t \frac{f_q(t)}{t^{p-1}} \tau^{p-1} \, d\tau
= \frac{f_q(t)}{t^{p-1}} (t^p - s^p).
\end{equation*}
Consequently, combining both inequalities, we obtain that
\begin{align*}
sf_q(s) - p F_q(s) &= sf_q(s) - p\int_0^s f_q(\tau) d\tau =  \frac{f_q(s)}{s^{p-1}} s^p - p F_q(t) + p \int_s^t f_q(\tau) \, d\tau \\
&< \frac{f_q(t)}{t^{p-1}} s^p - p F_q(t) + \frac{f_q(t)}{t^{p-1}} (t^p - s^p) \\
&= \frac{f_q(t)}{t^{p-1}} t^p - p F_q(t) = t f_q(t) - p F_q(t),
\end{align*}
where we conclude that $\xi$ is increasing in $(0,\infty)$. As condition \ref{f_4} holds true and $f_q$ is odd, then the primitive $F_q$ is even. Therefore, for any $t>0$ we have 
\[
\xi(-t)=(-t)f_q(-t) - p F_q(-t)= tf_q(t) - pF_q(t) = \xi(t),
\]
deducing that $\xi$ is also an even function and showing that it decreases in $(-\infty,0)$.

To conclude the last part of item \ref{Item:CrecimientoCombinado}, just observe that $f_q(0)=F_q(0)$, and so, $\xi(0)=0$.

\end{proof}

Using this result, we now prove that the Nehari set is nonempty and that any point $u\neq 0$ in $H_g^m(M)$ can be projected onto it. To this end, for any $u\in H^m_g(M)\smallsetminus\{0\}$, we define
\[
I_u:[0,\infty)\to \mathbb{R},\qquad  I_u(t):=I(tu).
\]
Notice that $I_u$ is of class $C^1$ and that $tu \in \mathcal{N} $  if and only if $I'_u(t)=0$. We have the following result

\begin{lemma}\label{Lemma:NehariNonEmpty}
For each $u \in H^{m}_{g}(M) \setminus \{0\}$, there exists a unique $t_u>0$ such that $t_u u \in \mathcal{N}$. Moreover, 
\[
\max_{t\in(0,\infty)} I_u(t)=I_u(t_u)>0
\]
  In particular, $\mathcal{N}\neq\emptyset$.
\end{lemma}
\begin{proof}
Let $u \in H^{m}_{g}(M) \setminus \{0\}$. We divide the proof into three steps.

\smallskip
\textbf{Step 1.} \emph{ There exists $t>0$ such that $tu\in\mathcal{N}$.}
\smallskip

We first need to prove the existence of a critical point for $I_u$; to do this, by Rolle's Theorem, it suffices to show that there exists a sufficiently small $t>0$  such that $I(u)>0$ and to show that $I_u(t)\to-\infty$ as $t\to\infty$. 

First, notice that  the continuity of the Sobolev embedding $H_g^m(M)\hookrightarrow L_g^s(M)$ for $2\leq s\leq 2_{m,d}^*$ yields the existence of a positive constant $C_1>0$ such that
\begin{equation}\label{Inequality:SobolevEmbedding}
\Vert u\Vert_s \leq C_1\Vert u\Vert_{\textbf{a}},\qquad \text{for }s=2,\ q\ \text{ and }\ r.
\end{equation}
Let $\varepsilon>0$ be such that
\begin{equation}\label{Definition:Epsilon}
1-C_1^2\varepsilon>0. 
\end{equation}
With this choice of $\varepsilon$, by condition \eqref{Crescimentoforigem}, there exists $C_\varepsilon>0$ such that
\[
\frac{\varepsilon}{2}\vert tu\vert^2 + \frac{C_\varepsilon}{q}\vert tu\vert^q\geq -F_q(tu)\geq -\frac{\varepsilon}{2}\vert tu\vert^2 - \frac{C_\varepsilon}{q}\vert tu\vert^q,\qquad \forall\;t>0.
\]
Using \eqref{Inequality:SobolevEmbedding}, for any $t>0$, we then get that
\begin{align}\label{Inequality:EstimateIntegralF1}
-\int_M F_q(tu) \;dV_g - \frac{1}{r}\int_M \vert tu\vert^r\; dV_g 
&\geq -\frac{\varepsilon}{2}\int_M \vert tu\vert^2\;dV_g - \frac{C_\varepsilon}{q}\int_M\vert tu\vert^q\; dV_g - \frac{1}{r}\int_M \vert tu\vert^r\; dV_g \nonumber \\
&\geq -\frac{\varepsilon C_1^2}{2}\Vert tu\Vert_{\textbf{a}}^2- \frac{C_1^qC_\varepsilon}{q}\Vert tu\Vert_{\textbf{a}}^q - \frac{C_1^r}{r}\Vert tu\Vert_{\textbf{a}}^r \nonumber\\
&=-\frac{\varepsilon C_1^2}{2}t^2\Vert u\Vert_{\textbf{a}}^2- \frac{C_1^qC_\varepsilon}{q}t^q\Vert u\Vert_{\textbf{a}}^q - \frac{C_1^r}{r}t^r\Vert u\Vert_{\textbf{a}}^r,
\end{align}

and also

\begin{align}\label{Inequality:EstimateIntegralF2}
-\int_M F_q(tu) \;dV_g - \frac{1}{r}\int_M \vert tu\vert^r\; dV_g  
&\leq \frac{\varepsilon}{2}\int_M \vert tu\vert^2 \;dV_g + \frac{C_\varepsilon}{q}\int_M \vert tu\vert^q \; dV_g - \frac{1}{r}\int_M \vert tu\vert^r\; dV_g \nonumber  \\
&\leq t^2\frac{C_1^2\varepsilon}{2}\Vert u\Vert_{\textbf{a}}^2 + t^q\frac{C_1^qC_\varepsilon}{q} \Vert u\Vert_{\textbf{a}}^q- \frac{t^r}{r}\int_M \vert u\vert^r\; dV_g 
\end{align}

Therefore, on the one hand, inequality \eqref{Inequality:EstimateIntegralF1} yields for every $t>0$ that
\begin{align*}
I_u(t) &= I(tu) = \frac{1}{2}\Vert tu\Vert_{\textbf{a}}^2 \pm \frac{1}{p}\int_M \vert\nabla (tu)\vert_g^p \;dV_g - \int_M F_q(tu) \;dV_g - \frac{1}{r}\int_M \vert tu\vert^r \; dV_g\\
&\geq \frac{t^2}{2} \Vert u\Vert_{\textbf{a}}^2 \pm \frac{t^p}{p}\int_M \vert\nabla u\vert_g^p \;dV_g  -\frac{\varepsilon C_1^2}{2}t^2\Vert u\Vert_{\textbf{a}}^2- \frac{C_1^qC_\varepsilon}{q}t^q\Vert u\Vert_{\textbf{a}}^q - \frac{C_1^r}{r}t^r\Vert u\Vert_{\textbf{a}}^r\\
&= t^2\left(\frac{1-C_1^2\varepsilon}{2}\right)\Vert u\Vert_{\textbf{a}}^2 \pm \frac{t^p}{p}\int_M \vert\nabla u\vert_g^p \;dV_g  - \frac{C_1^qC_\varepsilon}{q}t^q\Vert u\Vert_{\textbf{a}}^q - \frac{C_1^r}{r}t^r\Vert u\Vert_{\textbf{a}}^r\\
&= t^2\left[  \left(\frac{1-C_1^2\varepsilon}{2}\right)\Vert u\Vert_{\textbf{a}}^2 \pm \frac{t^{p-2}}{p}\int_M \vert\nabla u\vert_g^p \;dV_g  - \frac{C_1^qC_\varepsilon}{q}t^{q-2}\Vert u\Vert_{\textbf{a}}^q - \frac{C_1^r}{r}t^{r-2}\Vert u\Vert_{\textbf{a}}^r \right]
\end{align*}
As $2<p<q<r$, as $\Vert u\Vert_{\textbf{a}}\neq 0$, and since $1-C_1^2\varepsilon>0$  by \eqref{Definition:Epsilon}, the term in the brackets 
\[
\left(\frac{1-C_1^2\varepsilon}{2}\right)\Vert u\Vert_{\textbf{a}}^2 \pm \frac{t^{p-2}}{p}\int_M \vert\nabla u\vert_g^p \;dV_g  - \frac{C_1^qC_\varepsilon}{q}t^{q-2}\Vert u\Vert_{\textbf{a}}^q - \frac{C_1^r}{r}t^{r-2}\Vert u\Vert_{\textbf{a}}^r\quad \longrightarrow \quad\left(\frac{1-C_1^2\varepsilon}{2}\right)\Vert u\Vert_{\textbf{a}}^2>0
\]
as $t\to 0$. This implies the existence of $t_1>0$ small enough such that
\[
I(tu)>0\qquad\text{ for any }0<t\leq t_1.
\]

On the other hand, \eqref{Inequality:EstimateIntegralF2} yields
\begin{align}\label{Equation:AboveEstimateEnergyNehari}
I_u(t)&=\frac{1}{2}\|tu \|_\textbf{a}^{2}   \pm \frac{1}{p}\int_{M}| \nabla (tu) |_g^p \, dV_g -\int_{M} F_q(tu)  \, dV_g -\frac{1}{r}\int_{M}|tu| ^{r} \, dV_g \nonumber\\
&\leq \frac{t^2}{2}\|u \|_\textbf{a}^{2}   \pm \frac{t^p}{p}\int_{M}| \nabla u |_g^p \, dV_g + t^2\frac{C_1^2\varepsilon}{2}\Vert u\Vert_{\textbf{a}}^2 + t^q\frac{C_1^qC_\varepsilon}{q} \Vert u\Vert_{\textbf{a}}^q- \frac{t^r}{r}\int_M \vert u\vert^r\; dV_g \nonumber\\
&= t^r\left[ \frac{1}{t^{r-2}}\|u \|_\textbf{a}^{2}\left(\frac{1+C_1^2\varepsilon}{2}\right)   \pm \frac{1}{t^{r-p}}\frac{1}{p}\int_{M}| \nabla u |_g^p \, dV_g +  \frac{1}{t^{r-q}}\frac{C_1^qC_\varepsilon}{q} \Vert u\Vert_{\textbf{a}}^q- \frac{1}{r}\int_M \vert u\vert^r\; dV_g \right]
\end{align}
As $2<p<q<r$ and as $u\neq0$, then $-\int_M \vert u\vert^r\; dV_g<0$ and the term in the bracket
\[
\frac{1}{t^{r-2}}\|u \|_\textbf{a}^{2}\left(\frac{1+C_1^2\varepsilon}{2}\right)   \pm \frac{1}{t^{r-p}}\frac{1}{p}\int_{M}| \nabla u |_g^p \, dV_g +  \frac{1}{t^{r-q}}\frac{C_1^qC_\varepsilon}{q} \Vert u\Vert_{\textbf{a}}^q- \frac{1}{r}\int_M \vert u\vert^r\; dV_g\quad\longrightarrow\quad - \frac{1}{r}\int_M \vert u\vert^r\; dV_g <0,
\]
as $t\to\infty$, implying that
\[
I_u(tu)\to-\infty,\quad\text{ as }t\to\infty, 
\]
and this proves Step 1.

\medskip
\textbf{Step 2.} \emph{Uniqueness of $t\in(0,\infty)$.}
\smallskip

To show the  uniqueness of $t$, let us  suppose, by contradiction,  that there exist $t \neq s$  with $0<t <s $ satisfying  $tu,su\in\mathcal{N}$. Then, on the one hand, if $u(x)\neq 0$, then
\[
\frac{f_q(tu(x))tu(x)}{t^p} = \frac{f_q(tu(x))\vert u(x)\vert^p}{\vert tu(x) \vert^{p-2}(tu(x))}
\]
Hence, by the definition of $\mathcal{N}$ given in \eqref{DefinitionNehari}, for $tu$ we have that
\begin{align}\label{Equation:UniquenessProjection1}
&\frac{1}{t^{p-2}} \|u\|_{\textbf{a}}^{2}   \pm  \int_{M}|\nabla u|^{p}  \, dV_g =\frac{1}{t^p}\left[\Vert tu\Vert_{\textbf{a}}^2 \pm \int_{M}  |\nabla (tu) |^{p} \,  dV_g\right] = \int_{M} \frac{f_q(tu)t u}{t^p} \,  dV_g +\frac{1}{t^p}\int_{M}  |tu |^{r}   \,  dV_g \nonumber\\
&= \int_{\{u\neq 0\}} \frac{f_q(tu)t u}{t^p} \,  dV_g +t^{r-p}\int_{M}  |u |^{r}   \,  dV_g =  \int_{\{u\neq 0\}}\frac{f_q(tu)|u|^{p}}{|tu|^{p-2} (tu)} \,  dV_g +t^{r-p}\int_{M}  |u |^{r}   \,  dV_g
\end{align}
and analogously for $su$ we obtain
\begin{align}\label{Equation:UniquenessProjection2}
  \frac{1}{s^{p-2}} \|u\|_{\textbf{a}}^{2}   \pm  \int_{M}|\nabla u|^{p}  \, dV_g  = \int_{M} \frac{f_q(su)|u|^{p}}{|s u|^{p-2} (su)} \, dV_g +s^{r-p}   \int_{M}|u|^{r}\, dV_g
\end{align}
Subtracting \eqref{Equation:UniquenessProjection1} and \eqref{Equation:UniquenessProjection2}, we derive that
\begin{equation}\label{Equation:Uniqueness}
    \left(\frac{1}{t^{p-2}}-  \frac{1}{s^{p-2}}  \right)\|u\|_{\textbf{a}}^{2} =\int_{M} \left( \frac{f_q(tu)}{|t u|^{p-2} (tu)} -\frac{f_q(su)}{|s u|^{p-2} (su)}\right)|u|^{p} \, dV_g +\left( t^{r-p}-s^{r-p} \right) \int_{M}|u|^{r}\, dV_g
\end{equation}
 Since $u\neq 0$, $2<p$ and $0<t<s$, then $1/t^{p-2}- 1/s^{p-2}>0$ and the left-hand side of this identity is positive. However, we claim that the right-hand side of the identity is negative. To see this, first notice that $(t^{r-p}- s^{r-p} )<0$ because $p<r$ and $0<t<s$, and as $u\neq0$, the second integral on the right-hand side of identity \eqref{Equation:Uniqueness} is negative. Moreover, by \ref{f_4}, the map 
\[
t\mapsto \frac{f_q(t)}{\vert t\vert^{p-2}t}
\] 
is even, and by \ref{f_3}, it is increasing in $(0,\infty)$. Hence,
\[
\frac{f_q(tu)}{|t u|^{p-2} (tu)} -\frac{f_q(su)}{|s u|^{p-2} (su)}\leq 0 \qquad\text{ if }\  u>0
\]
and also
\[
\frac{f_q(tu)}{|t u|^{p-2} (tu)} -\frac{f_q(su)}{|s u|^{p-2} (su)}= \frac{f_q(-tu)}{|-t u|^{p-2} (-tu)} -\frac{f_q(-su)}{|-s u|^{p-2} (-su)}\leq 0 \qquad\text{ if }\   u<0
\]
because in this last case $0<t<s$ implies that $0<-tu<-su$. Thus, the first integral on the right-hand side of identity \eqref{Equation:Uniqueness} is negative. This proves the claim, obtaining a contradiction. Therefore, there exists a unique $t\in(0,\infty)$ such that $tu\in\mathcal{N}$, which we will denote $t_u$, concluding the proof of Step 2.

\medskip
\textbf{Step 3.} $t_u$ is the unique critical point of $I_u$ and $\max_{t\in(0,\infty)} I_u(t):=I(t_u)$.

\smallskip
This follows immediately from Steps 1 and 2. Indeed, we have already observed that $t\in(0,\infty)$ is a critical point of $I_u$ if and only if $tu\in\mathcal{N}$; thus, Step 2 actually shows the uniqueness of the critical point. Moreover, as $t_u\in(0,\infty)$ is the unique critical point of $I_u$ in $(0,\infty)$, and because  $I_u(0)=0$, $I_u(t)> 0$ for $0<t\leq t_1$ and $I_u(t)\rightarrow-\infty$ by Step 1, this implies that $I_u$ attains its maximum at $t_u$.
\end{proof}

\begin{remark}
We emphasize that \emph{ the projection onto $\mathcal{N}$,} 
\[
u\mapsto t_uu,\qquad u\neq 0,
\]
depends only on \ref{f_1}, \ref{f_2} and on inequality $2<p<q<r$. On the other hand, conditions \ref{f_3} and \ref{f_4} were essential for the uniqueness of the projection.
\end{remark}

Next, we prove that $\mathcal{N}$ is closed, bounded away from zero, and that $I$ is bounded from below on this set.

\begin{lemma}\label{Lemma:NehariClosed}
The following statements hold true.
\begin{enumerate} [label=(\roman*)]
\item\label{Item:BoundedBelow} There exists a constant $K_0>0$ such that 
\[
 \| u\|_{\textbf{a}}\geq K_0 \qquad \text{and}\qquad  I(u)\geq K_0,\quad \forall\; u\in\mathcal{N}.
\] 
\item \label{Item:NehariClosed} $\mathcal{N} \subset H_g^m(M)$ is closed.
\end{enumerate}
\end{lemma}
\begin{proof}
To prove \ref{Item:BoundedBelow}, by definition of $\mathcal{N}$, for any $u\in\mathcal{N}$ we have that
\begin{align}\label{Equation:NehariBoundedCases}
  \|u\|^{2}_{\textbf{a}} 
 =  \mp \int_{M} |\nabla u|^{p} \,  dV_g  + \int_{M}f_q(u)u \,  dV_g   +  \int_{M}|u|^{r} \,  dV_g.
\end{align}
The proof is divided into two cases.

\textbf{Case 1}: If we consider the minus sign in the right-hand side of identity \eqref{Equation:NehariBoundedCases}, taking $C_1>0$ as in \eqref{Inequality:SobolevEmbedding} and $\varepsilon>0$ as in \eqref{Definition:Epsilon}, we have, from the Sobolev inequality \eqref{Inequality:SobolevEmbedding} and the property \eqref{crescimentoforigem} of $f_q$, that
\begin{align*}
 \|u\|^{2}_{\textbf{a}} 
& =  \mp \int_{M} |\nabla u|^{p} \,  dV_g  + \int_{M}f_q(u)u \,  dV_g   +  \int_{M}|u|^{r} \,  dV_g\leq \varepsilon\int_{M}|u|^{2} \,  dV_g + C_{\varepsilon} \int_{M}|u|^{q} \,  dV_g   +  \int_{M}|u|^{r} \,  dV_g\\
& \leq  \varepsilon C_1^2  \|u\|^{2}_{\textbf{a}}+ C_\varepsilon C_1^q  \|u\|^{q}_{\textbf{a}}+ C_1^r  \|u\|^{r}_{\textbf{a}},
\end{align*}
from which we obtain that
\begin{equation}\label{Equation:NehariBoundedCasesSubcases}
0<(1-\varepsilon C_1^2)\Vert u\Vert_{\textbf{a}} \leq C_\varepsilon C_1^q  \|u\|^{q}_{\textbf{a}}+ C_1^r  \|u\|^{r}_{\textbf{a}}
\end{equation}
If $\Vert u\Vert_{\textbf{a}}\geq 1$, as $2<q<r$, then
\[
(1-\varepsilon C_1^2)\Vert u\Vert_{\textbf{a}}^2 \leq (C_\varepsilon C_1^q + C_1^r)\Vert u\Vert_{\textbf{a}}^r, 
\]
and, therefore
\[
\Vert u\Vert_{\textbf{a}}\geq K_1^{\frac{1}{r-2}}>0, \qquad \text{where }\quad K_1:= \frac{1-\varepsilon C_1^2}{C_\varepsilon C_1^q + C_1^r} >0.
\]
While if $\Vert u\Vert_{\textbf{a}}\leq 1$, as $2<q<r$, we have that
\[
(1-\varepsilon C_1^2)\Vert u\Vert_{\textbf{a}}^2 \leq (C_\varepsilon C_1^q + C_1^r)\Vert u\Vert_{\textbf{a}}^q, 
\]
implying that
\[
\Vert u\Vert_{\textbf{a}}\geq K_1^{\frac{1}{q-2}}>0
\]
in this case. Hence, taking $\widetilde{K}_1:=\min\{K_1^{\frac{1}{q-2}},K_1^{\frac{1}{r-2}}\}$, we infer that $\Vert u\Vert_{\textbf{a}}\geq \widetilde{K}_1$ for every $u\in\mathcal{N}$ in Case 1.

\textbf{Case 2}: If we consider the plus sign in the right-hand side of identity \eqref{Equation:NehariBoundedCases}, Proposition \ref{Proposition:CriticalGN} imply the existence of a constant $C_2>0$ such that
\begin{equation}\label{Equation:EstimatesP-NormGradient}
\int_M\vert \nabla u\vert^p \;dV_g \leq C_2 \Vert u\Vert_{\textbf{a}}^{\Upsilon p}\Vert u\Vert_2^{p(1-\Upsilon)}\leq C_2 (\min_{M}a_0)^{-p(1-\Upsilon)/2} \Vert u\Vert_{\textbf{a}}^{\Upsilon p}\Vert u\Vert_{\textbf{a}}^{p(1-\Upsilon)}=C_2 (\min_{M}a_0)^{-p(1-\Upsilon)/2} \Vert u\Vert_{\textbf{a}}^p.
\end{equation}
Arguing as in Case 1, using \eqref{Equation:EstimatesP-NormGradient}, the Sobolev inequality \eqref{Inequality:SobolevEmbedding}, \eqref{crescimentoforigem} with the same $\varepsilon$ as  before, we get from \eqref{Equation:NehariBoundedCases} that
\begin{align*}
0<(1-\varepsilon C_1^2)\Vert u\Vert_{\textbf{a}}\leq C_\varepsilon C_1^q \Vert u\Vert_{\textbf{a}}^q + C_1^r \Vert u\Vert_{\textbf{a}}^r + C_2 (\min_{M}a_0)^{-p(1-\Upsilon)/2}\Vert u\Vert_{\textbf{a}}^p.
\end{align*}
As $2<p<q<r$, arguing as in Case 1, if 
\[
K_2:=\frac{1-\varepsilon C_1^2}{C_\varepsilon C_1^q + C_1^r+C_2 (\min_{M}a_0)^{-p(1-\Upsilon)/2}}>0,
\]
then 
\[
\Vert u\Vert_{\textbf{a}}\geq
\begin{cases}
K_2^{\frac{1}{p-2}} & \text{if} \quad \Vert u\Vert_{\textbf{a}}\leq 1\\
K_2^{\frac{1}{r-2}} & \text{if} \quad \Vert u\Vert_{\textbf{a}}\geq 1.
\end{cases}
\]
So, taking $\widetilde{K_2}:=\min\{ K_2^{\frac{1}{p-2}}, K_2^{\frac{1}{r-2}}\}$ finishes the proof of Case 2.

\smallskip

We now prove the second assertion of \ref{Item:BoundedBelow}, i.e., that the functional $I$ is bounded from below. Take $K_0':=\min\{\widetilde{K}_1,\widetilde{K}_2\}>0$. Then, from the preceding analysis, we deduce that
\begin{equation}\label{BoundBothCases}
\Vert u\Vert_{\textbf{a}}\geq K_0',\qquad\forall\;u\in\mathcal{N}
\end{equation}
independently of the sign of the $L^p$ term of the gradient in the definition of $I$. Hence,  since that $2<p<q<r$, by \eqref{crescente} in Lemma \ref{Lemma:Properties-f}, together with identity \eqref{Equation:FunctionalDefinedOnNehari} and inequality \eqref{BoundBothCases},  for any $u\in\mathcal{N}$ we have that

\begin{equation}\label{Equality:I-I'/p}
\begin{split}
 I(u)  
   &=   \left(\frac{p-2}{2p}\right)\|u\|_{\textbf{a}}^{2}+ \frac{1 }{p}  \int_{M}  (f_q(u)u-pF_q(u)) \, dV_g + \left( \frac{r-p }{rp}\right) \int_{M} |u|^{r} \, dV_g \\
   &\geq   \left(\frac{p-2}{2p}\right)\|u\|_{\textbf{a}}^{2}
   \geq \left(\frac{p-2}{2p}\right)K_0'.
   \end{split}
 \end{equation}
Therefore, to end the proof of item \ref{Item:BoundedBelow}, it suffices to take 
\[
K_0:=\min\left\{K_0',\left(\frac{p-2}{2p}\right)K_0'\right\}.
\]

\bigskip

Now we prove item \ref{Item:NehariClosed}. Define the functional $G: H_g^m(M) \to \mathbb{R}$ given by

 \[
 G(u) = I'(u)u = \|u\|_{\textbf{a}}^{2} \pm \int_{M} |\nabla u|^{p} \,  dV_g -  \int_{M}f_q(u)u \, dV_g-\int_{M} |u|^{r} \, dV_g,
 \]

Since $I$ is of class $C^1$, this functional is continuous.
In particular, $G^{-1}(\{0\})$ is closed in $H_g^m(M)$. Hence, as the norm $\Vert\cdot\Vert_{a}$ is also continuous, we have from item \ref{Item:BoundedBelow} and the definition of the Nehari set \eqref{DefinitionNehari} that
\[
\mathcal{N}:=\left\{u\in H_g^m(M)\;:\;u\neq 0 , I'(u)u=0\right\} = \Vert\cdot\Vert_{\textbf{a}}^{-1}\big([K_0,\infty)\big)\cap G^{-1}(\{0\}),
\]
proving that $\mathcal{N}$ is closed in $H_g^m(M)$.
\end{proof}

 As a consequence of item \ref{Item:BoundedBelow} in the previous lemma, we have that
\begin{align}\label{Definition:InfimumNehari}
   c_0= \inf_{\mathcal{N}}I>0.
\end{align}
\begin{definition}
A nontrivial weak solution $u$ to \eqref{Problem:Main} is a \emph{ground state} if $I(u)=c_0$
\end{definition}

For any $u\in H_g^m(M)\smallsetminus\{0\}$, let $t_u>0$ be the unique element in $(0,\infty)$ such that $t_uu\in\mathcal{N}$, as given in Lemma \ref{Lemma:NehariNonEmpty}. Then we can define the projection onto the Nehari set 
\[
\widehat{H}: H_g^m(M)\smallsetminus\{0\}\rightarrow\mathcal{N},\qquad  \widehat{H}(u):= t_u u
\]
We now show that $\widehat{H}$ induces a homeomorphism between the sphere in $H_g^m(M)$,
\[
S_1(0):=\{u\in H_g^m(M)\;:\; \Vert u\Vert_{\textbf{a}}=1\} 
\]
and the Nehari set $\mathcal{N}$. We first establish the following lemma. Recall that a subset $\mathcal{A}$ in a Banach space $E$ is  symmetric about the origin if $u\in\mathcal{A}$ if and only if $-u\in\mathcal{A}$, and that a map $\varphi:E\rightarrow E$ is odd if  $\varphi(-u)=-\varphi(u)$.

\begin{lemma} \label{Lemma:ProjectionOdd}
$\mathcal{N}$ is symmetric about the origin and the map $\widehat{H}$ is continuous and odd.
\end{lemma}

\begin{proof}

To verify that $\mathcal{N}$ is symmetric about the origin, let $u\in\mathcal{N}$ and consider $-u\neq0$. Since $f_q$ is odd by assumption \ref{f_4}, we have that
\[
\begin{split}
\Vert - u\Vert_{\textbf{a}}^2 \pm \int_{M}  |\nabla(- u) |^{p} \,  dV_g & = \Vert  u\Vert_{\textbf{a}}^2 \pm \int_{M}  |\nabla u |^{p} \,  dV_g =\int_{M} f_q(u) u \,  dV_g +\int_{M}  |u |^{r}   \,  dV_g \\
&=   \int_{M} f_q(-u)(- u) \,  dV_g +\int_{M}  |-u |^{r}   \,  dV_g
\end{split}
\]
and $-u\in\mathcal{N}$ by \eqref{DefinitionNehari}.

From this, we also conclude that $\widehat{H}$ is odd. Indeed, if $u\in H_g^m(M)$, then $t_u u\in\mathcal{N}$ if and only if $-t_u u = t_u(-u)\in\mathcal{N}$, and by the uniqueness of the projection given in Lemma \ref{Lemma:NehariNonEmpty}, we infer that
\begin{equation}
t_{-u} = t_{u},\qquad\forall\; u\in H_g^m(M)\smallsetminus\{0\}.
\end{equation}
Therefore, for any $u\in H_g^m(M)\smallsetminus\{0\}$ 
\[
\widehat{H}(-u) = t_{-u}(-u)=-t_u u = -\widehat{H}(u).
\]

\medskip

To prove the continuity, it suffices to show that the map $u\mapsto t_u$ is continuous. Let $u\in H_g^m(M)\smallsetminus\{0\}$ and let $(u_n)$ be a sequence in $H_g^m(M)$ such that $u_n\to u$ in $H_g^m(M)$.  For convenience, denote $t_n:= t_{u_n}$. Then, passing to a subsequence if necessary, we can suppose that $\Vert u_n\Vert_{\textbf{a}}\leq \Vert u\Vert_{\textbf{a}}/ 2$. Hence, as $t_nu_n\in\mathcal{N}$, by item \ref{Item:BoundedBelow} in Lemma \ref{Lemma:NehariClosed}, we have that 
\begin{equation}\label{Equation:SequenceBoundedBelow}
t_n\geq \frac{K_0}{\Vert u_n\Vert_{\textbf{a}}}\geq \frac{2K_0}{\Vert u\Vert_{\textbf{a}}}>0,
\end{equation}
so, up to a subsequence, $(t_n)$ is bounded away from zero. 

Next, we see that the sequence $(t_n)$  is also bounded from above. Suppose, in order to get a contradiction, that the sequence is not bounded, so that $t_n\to\infty$ as $n\to\infty$ in a subsequence, which we denote the same. Since $u_n\to u$ in $H_g^m(M)$, then, by the Sobolev inequality \eqref{Inequality:SobolevEmbedding} and Proposition \ref{Proposition:CriticalGN}, we have that
\[
\begin{split}
\Vert u_n\Vert_{\textbf{a}}^s &= \Vert u\Vert_{\textbf{a}}^s + o(1), \text{ for }s=2,q,\\  \quad \int_M \vert u_n\vert^r \; dV_g &= \int_M \vert u\vert^r \; dV_g + o(1) \quad\text{ and}\\
\int_M \vert\nabla u_n\vert^p_g\; dV_g &= \int_M \vert\nabla u\vert^p_g\; dV_g + o(1),
\end{split}
\]
where $o(1)\to0$ as $n\to\infty$. Hence, as $2<p<q<r$, by estimate \eqref{Equation:AboveEstimateEnergyNehari}, we have that
\begin{align*}
I(t_nu_n)
&\leq t_n^r\left[ \frac{1}{t_n^{r-2}}\|u_n \|_\textbf{a}^{2}\left(\frac{1+C_1^2\varepsilon}{2}\right)   \pm \frac{1}{t_n^{r-p}}\frac{1}{p}\int_{M}| \nabla u_n |_g^p \, dV_g +  \frac{1}{t_n^{r-q}}\frac{C_1^qC_\varepsilon}{q} \Vert u_n\Vert_{\textbf{a}}^q- \frac{1}{r}\int_M \vert u_n\vert^r\; dV_g \right]\\
&= t_n^r \left[ \frac{C}{t_n^{r-2}} ( \|u  \|_\textbf{a}^{2} + o(1))   \pm \frac{C}{t_n^{r-p}}\left(\int_{M}| \nabla u |_g^p \, dV_g+ o(1)\right)  +  \frac{C}{t_n^{r-q}} (\Vert u\Vert_{\textbf{a}}^q + o(1))- \frac{1}{r}\left(\int_M \vert u\vert^r\; dV_g + o(1)\right) \right]\\
&= t_n^r \left[ - \frac{1}{r}\int_M \vert u\vert^r\; dV_g + o(1) \right]
\end{align*}
where $C>0$ denotes a positive constant, not necessarily the same one. As $t_n\to\infty$, this implies that $I(t_nu_n)<0$ for $n$ sufficiently large, which is a contradiction to the fact that $I(t_nu_n)>0$ stated in item \ref{Item:BoundedBelow} of Lemma \ref{Lemma:NehariClosed}.  Hence, the sequence $(t_n)$ is bounded in $\mathbb{R}$, say, by a positive constant $B$. Therefore, if $A:=\frac{2K_0}{\Vert u\Vert_{\textbf{a}}}$ is the constant given in \eqref{Equation:SequenceBoundedBelow}, we have that $t_n\in [A,B]\subset(0,\infty)$ for every $n\in\mathbb{N}$. Consequently, up to another subsequence, which we also denote the same, there exists $t_\ast\in[A,B]$ such that $t_n\to t_\ast$ as $n\to\infty$ and $t_nu_n\to t_\ast u$ in $H_g^m(M)$. Since $\mathcal{N}$ is closed by Lemma \ref{Lemma:NehariClosed} and since $t_nu_n\in\mathcal{N}$, we conclude that $t_\ast u\in\mathcal{N}$. But $t_u>0$ is the unique element in $(0,\infty)$ such that $t_uu\in\mathcal{N}$, concluding that $t_\ast=t_u$. Therefore, $t_n:=t_{u_n}\to t_u$ in a subsequence. As the sequence $(u_n)$ was arbitrary, we infer that the map $u\mapsto t_u$ is continuous, as we wanted to show.
\end{proof}

\begin{corollary}\label{Corollary:Homeomorphism}
The map
\[
H:=\widehat{H}\big\vert_{S_1(0)}:S_1(0)\rightarrow\mathcal{N}
\]
is an odd homeomorphism with odd inverse $H^{-1}(u):=\frac{u}{\Vert u\Vert_{\textbf{a}}}$.
\end{corollary}

\begin{proof}
It is clear that the sphere $S_1(0)$ is symmetric about the origin. By the previous lemma, the map $H$ is odd and continuous. Moreover, since $\mathcal{N}$ is bounded away from zero by Lemma \ref{Lemma:NehariClosed}, it is readily seen that the map $H^{-1}$ is also odd and continuous. Moreover, on the one hand, if $u\in S_1(0)$, we have that
\[
H^{-1}\circ H(u)= H^{-1}(t_uu)= \frac{t_uu}{\Vert t_u u \Vert_{\textbf{a}}} = \frac{u}{\Vert u\Vert_{\textbf{a}}} = u.
\]
On the other hand, if $u\in\mathcal{N}$, then $t_u=1$. Observe that $u=\Vert u\Vert_{\textbf{a}}\frac{u}{\Vert u\Vert_{\textbf{a}}}$; hence, the unique $t\in (0,\infty)$ such that $t\frac{u}{\Vert u\Vert_{\textbf{a}}}\in\mathcal{N}$ is $t=\Vert u\Vert_{\textbf{a}}.$ Therefore
\[
H\circ H^{-1}(u)= H\left( \frac{u}{\Vert u\Vert_{\textbf{a}}} \right) = \Vert u\Vert_{\textbf{a}}\frac{u}{\Vert u\Vert_{\textbf{a}}} = u.
\]
This completes the proof.
\end{proof}

We end this section with a useful convergence result.

\begin{lemma}
\label{Lemma:ConvergenceFuntion}
Let $(u_n)$ be a bounded sequence in $H_g^m(M)$ and $u_0\in H_g^m(M)$ be such that
\[
u_n\to u_0 \qquad\text{ in }\ L_g^s(M),\  s=2,q.
\]
Then, 
\begin{align*}
    \lim_{n \to + \infty} \int_{M}f_q(u_{n})u_{n}\, dV_g =  \int_{M}f_q(u)u \, dV_g 
\end{align*}
in a subsequence.
\end{lemma}

\begin{proof}
First, observe that that $L_g^q(M)\hookrightarrow L_g^2(M)$ because $M$ is compact. Hence, as $u_n\to u_0$ in $L_g^q(M)$ and in $L_g^2(M)$, then, passing to a subsequence, we have that $u_n(x)\to u_0(x)$ a.e. in $M$ and there exists a function $h\in L_g^q(M)\cap L_g^2(M)$ such that $\vert u(x)\vert \leq h(x)$ a.e. in $M$. Then, for a fixed $\varepsilon>0$, identity \eqref{crescimentoforigem} yields that, 
\[
\vert f_q(u_n(x))u_n(x)\vert \leq \varepsilon \vert u_n(x)\vert^2 + C_\varepsilon\vert u_n(x)\vert^q\leq \varepsilon h^2(x) + C_\varepsilon h^q(x) \quad \text{a.e. in }M.
\]
As $\varepsilon h^2+C_\varepsilon h^q\in L^1_g(M)$ and since by continuity of $f_q$ we have that
\[
f_q(u_n(x)) u_n(x)\to f_q(u(x)) u(x)\qquad\text{ a.e. in }M, 
\]
by the Lebesgue's Dominated Convergence Theorem, we deduce that
\[
    \lim_{n \to + \infty} \int_{M}f_q(u_{n})u_{n}\, dV_g =  \int_{M}f_q(u)u \, dV_g 
\]
and the proof follows.
\end{proof}


\section{Variational setting with symmetries} \label{Section:SymmetricSetting}

Throughout this section, we fix a closed subgroup $\Gamma$ of $\text{Isom}(M,g)$ satisfying \eqref{ActionNotTransitive}, and assume that the function $a_0\in C^\infty(M)$, appearing in the definition of the operator $P_{\textbf{a}}^m$ \eqref{Def:Operator}, is $\Gamma$-invariant. The group $\Gamma$  acts on $H_{g}^{m}(M)$ in the usual way:
\begin{align*}
    \Gamma \times H_{g}^{m}(M) \to H_{g}^{m}(M), \quad (\gamma, u)\mapsto \gamma u:=u \circ \gamma^{-1}.
\end{align*}

The space of fixed points under this action coincides with the subspace of $\Gamma$-invariant functions 
\[
H_{g}^{m}(M)^{\Gamma}:=  \left\lbrace u \in H^{m}_{g}(M):\,\, u \,\, \mbox{is}\,\,\Gamma-\mbox{invariant}  \right\rbrace,
\]
which is closed in $H_g^m(M)$. 
Moreover, if $C^{\infty}(M)^{\Gamma}$ denotes the space of smooth $\Gamma$-invariant functions, $H_{g}^{m}(M)^{\Gamma}$ coincides with the closure of this space under the Sobolev norm $\Vert \cdot\Vert_{H_g^m(M)}$. As $M$ is compact, it follows from the Myers-Steenrod Theorem that $\text{Isom}(M,g)$ is compact (see, for instance, Section II.1 in \cite{KobayashiBook}). Since $\Gamma$ is closed in $\text{Isom}(M,g)$, $\Gamma$ is compact and admits a Haar measure. Together with  \eqref{ActionNotTransitive}, this implies the existence of $\Gamma$-invariant partitions of unity (see \cite{Palais1961}), and hence, the spaces $C^{\infty}(M)^{\Gamma}$ and $H_{g}^{m}(M)^{\Gamma}$ are infinite-dimensional. Observe that we can recover the non-symmetric setting by taking $\Gamma=\{Id\}$, because, in this case, $H_g^m(M)=H_g^m(M)^{\{Id\}}$.  

For any $\gamma\in\text{Isom}(M,g)$, the action of $\Gamma$ on $H_g^m(M)$ induces linear maps
\begin{align*}
\gamma:  H_{g}^{m}(M) \to H_{g}^{m}(M), \quad u \to u \circ \gamma^{-1}.
\end{align*}
Clearly, each map is a linear isomorphism. We verify that they are isometric isomorphisms, endowing $H_g^m(M)$ with the norm $\Vert\cdot \Vert_{\textbf{a}}$.

\begin{lemma}\label{Lemma:P_gCommutes}
 For every  $\gamma\in \text{Isom}(M,g)$ and every $u \in C^{\infty}(M)$,
 \begin{align}\label{Equation:P_gCommutes}
     P_{\textbf{a}}^m(u \circ \gamma)=(P_{\textbf{a}}^mu)\circ \gamma
 \end{align}
 In particular, if $a_0\in C^\infty(M)$ is positive and $\Gamma$-invariant, then
 \begin{align*}
\gamma:  (H_{g}^{m}(M),\Vert\cdot\Vert_{\textbf{a}}) \to (H_{g}^{m}(M),\Vert\cdot\Vert_{\textbf{a}})
\end{align*}
is a linear isometry, i.e.,
\[
\langle u\circ\gamma,v\circ\gamma\rangle_{\textbf{a}}= \langle u,v\rangle_{\textbf{a}},\qquad\forall\;u,v\in H_g^m(M)\text{ and }\forall\; \gamma\in\Gamma.
\]
\end{lemma}

\begin{proof}
For any $\gamma\in \text{Isom}(M,g)$ and any $u\in C^\infty(M)$, Lemma 4.3.3 in \cite{HebeyBook1997} yields that
\[
\Delta_g^j(u\circ\gamma) = (\Delta_g^j u)\circ\gamma,\qquad \forall \; j=0,1,\ldots,m. 
\]
This immediately implies \eqref{Equation:P_gCommutes}. 

Now, for any $\gamma\in\text{Isom}(M,g)$ and any $u,v\in C^{\infty}(M)$, using \eqref{Equation:P_gCommutes} together with the Change of Variables formula (see, for instance, \cite[Theorem 4.12]{HebeyBook1997}) and the fact that  $a_0$ is $\Gamma$-invariant, it follows that
\[
 \langle u\circ\gamma^{-1},  v\circ\gamma^{-1}\rangle_{\textbf{a}} = \int_M P_{\textbf{a}}^m(u\circ\gamma^{-1})v\circ\gamma^{-1}\; dV_g = \int_M [P_{\textbf{a}}^m(u)v]\circ\gamma^{-1} \; dV_g = \int_M P_{\textbf{a}}^m(u)v \; dV_g = \langle u, v\rangle_{\textbf{a}}.
\]
 As $C^\infty(M)$ is dense in $H_g^m(M)$, this yields that $\gamma$ extends uniquely to an isometry on $H_g^m(M)$.
\end{proof}

We say that a function $J:H_g^m(M)\to\mathbb{R}$ is $\Gamma$-invariant if  $J(\gamma u)=J(u)$ for every $\gamma\in\Gamma$ and every $u\in H_g^m(M)$. We next show that the energy functional $I$ is $\Gamma$-invariant.

\begin{lemma}\label{Lemma:FunctionalGammaInvariant}
If $a_0$ is $\Gamma$-invariant, then the functional $I$ is $\Gamma$-invariant.
\end{lemma}

\begin{proof}
First, notice that for any $u\in C^\infty(M)$ and any $\gamma\in \text{Isom}(M,g)$
\[
\nabla(u\circ\gamma^{-1}) = d\gamma(\nabla u\circ\gamma^{-1})
\]
where $d\gamma$ denotes the differential of $\gamma$. Hence, as $\gamma$ is an isometry, this implies that
\[
\vert \nabla(u\circ\gamma^{-1})\vert_g =\left( \langle  d\gamma(\nabla u\circ\gamma^{-1}) ,  d\gamma(\nabla u\circ\gamma^{-1})\rangle_g \right)^{1/2}= \left( \langle \nabla u\circ\gamma^{-1} ,  \nabla u\circ\gamma^{-1}\rangle_g \right)^{1/2} = \vert \nabla u\vert_g\circ\gamma^{-1}.
\]

Therefore, by the Change of Variables formula and Lemma \ref{Lemma:P_gCommutes}, for any $u\in C^\infty(M)$ and any $\gamma\in\Gamma$ we conclude that
\begin{align*}
I(\gamma u)& = I(u\circ\gamma^{-1}) 
= \frac{1}{2}\Vert u\circ\gamma^{-1}\Vert_{\textbf{a}} \pm \frac{1}{p}\int_M \vert\nabla (u\circ\gamma^{-1})\vert_g^p \;dV_g- \int_M F_q(u\circ\gamma^{-1})\;dV_g - \frac{1}{r}\int_M \vert u\circ\gamma^{-1}\vert^r \; dV_g\\
&= \frac{1}{2}\Vert u\Vert_{\textbf{a}} \pm \frac{1}{p}\int_M \vert\nabla u\vert_g^p\circ\gamma^{-1} \;dV_g- \int_M F_q(u)\circ\gamma^{-1}\;dV_g - \frac{1}{r}\int_M \vert u\vert^r \circ\gamma^{-1}\; dV_g\\
&= \frac{1}{2}\Vert u\Vert_{\textbf{a}} \pm \frac{1}{p}\int_M \vert\nabla u\vert_g^p \;dV_g- \int_M F_q(u)\;dV_g - \frac{1}{r}\int_M \vert u\vert^r \; dV_g\\
&= I(u).
\end{align*}
Since $C^\infty(M)$ is dense in $H_g^m(M)$, we deduce that $I$ is $\Gamma$-invariant, as we wanted to show.
\end{proof}

As a consequence of the last result, by the Principle of Symmetric Criticality \cite{Palais1979}, $\Gamma$-invariant solutions to the problem \eqref{Problem:Main} correspond to the critical points of $I$ restricted to the space $H_g^m(M)^\Gamma$, and the nontrivial ones lie in the $\Gamma$-invariant Nehari set
\[
\mathcal{N}^\Gamma:=\mathcal{N}\cap H_g^m(M)^\Gamma.
\]

Thanks to Lemma \ref{Lemma:NehariClosed}, $\mathcal{N}^\Gamma$ is closed and bounded away from zero, and we can define
\begin{align}\label{Eq:InfimumNehariGammaInvariant}
c_{0}^{\Gamma}:=\inf_{u\in\mathcal{N}^\Gamma} I (u)>0
\end{align}

Notice that a $\Gamma$-invariant solution to \eqref{Problem:Main} attaining this value will have the least energy among all other $\Gamma$-invariant solutions. Also, remark that $c_0=c_0^{\{Id\}}$ and, in this case, a $\{Id\}$-invariant solution attaining this $c_0^{\{Id\}}$ is a ground state solution. 

\medskip

The next result is classical and states that symmetries restore compactness (see \cite{HebeyVaugon2} for the case $m=1$ and, for instance, \cite[Lemma 2.2]{FernandezPalmasTorres2024} for $m\geq 2$).

\begin{lemma}\label{Lemma:HebeyVaugon}
    If $\alpha:=\min \{ \dim\; \Gamma x :\,\, x \in M  \}$, then 
    \begin{align*}
        H_{g}^{m}(M)^{\Gamma} \hookrightarrow L_{g}^{r}(M)
    \end{align*}
    is continuous for every $2<r \leq 2^{\ast}_{m,d-\alpha}:= \frac{2(d-\alpha)}{(d-\alpha)-2m}$ and  compact if $2<r < 2^{\ast}_{m,d-\alpha}=\frac{2(d-\alpha)}{(d-\alpha)-2m}$.
\end{lemma}

Observe that, $2_{m,d-\alpha}^\ast<2_{m,d}^\ast$ when $\alpha\geq 1$. Hence, with this result, we have that the Sobolev embedding
\[
H_g^m(M)^\Gamma\hookrightarrow L_g^r(M) \text{ is compact}
\]
when $r<2_{m,d}^\ast$ is considered without any symmetry assumption or when $r=2_{m,d}^\ast$ and the condition \eqref{Hyp:Gamma} are satisfied. For $c\in\mathbb{R}$, we say that a sequence $(u_n)$ in $H_g^m(M)^\Gamma$ is a $(PS)_c^\Gamma$-sequence  for the functional $I$ if $I (u_n)\rightarrow c$ and $I '(u_n)\rightarrow 0$ in $(H_g^m(M))'$, where $(H_g^m(M))'$ is the dual space of $H_g^m(M)$.

\begin{lemma}\label{BS}
If $(u_{n})$ is a $(PS)_c^\Gamma$ sequence for $I$ with $u_n\in\mathcal{N}^{\Gamma}$ for every $n\in\mathbb{N}$, then it is bounded in  $H^{m}_{g}(M)^\Gamma$.
\end{lemma}

\begin{proof}
 As $I(u_n)\to c$,  the sequence $(I(u_n))$ is bounded in $\mathbb{R}$ by a constant $C>0$. But, in addition, since $u_n\in\mathcal{N}^\Gamma\subset\mathcal{N}$, as in \eqref{Equality:I-I'/p} we have that
\begin{align*}
 C&\geq \vert I(u_n)\vert = I(u_n) \\
   &=   \left(\frac{p-2}{2p}\right)\|u_n\|_{\textbf{a}}^{2}+ \frac{1 }{p}  \int_{M}  (f_q(u_n)u_n-pF_q(u_n)) \, dV_g + \left( \frac{r-p}{rp}\right) \int_{M} |u_n|^{r} \, dV_g \\
   &\geq   \left(\frac{p-2}{2p}\right)\|u_n\|_{\textbf{a}}^{2},
\end{align*}
for every $n\in\mathbb{N}$. Since $(p-2)/2p>0$, the sequence $(u_n)$ is bounded in $H_g^m(M)^\Gamma$.
\end{proof}

To end this section, we now prove that the Palais-Smale condition holds true for $(PS)_c^\Gamma$-sequences in $\mathcal{N}^\Gamma$.

\begin{proposition} \label{Propo:Palais-Smale}
Suppose that $p$ satisfies condition \eqref{Equation:p_subcritical}. If either 
\begin{itemize}
\item $r<2_{m,d}^\ast$,
\end{itemize}
or
\begin{itemize}
\item $r=2_{m,d}^\ast$ and condition \eqref{Hyp:Gamma} hold true,
\end{itemize}
then, for any $c\in\mathbb{R}$, every $(PS)_c^\Gamma$-sequence $(u_n)$ for the functional $I$ with $u_n\in\mathcal{N}$, has a converging subsequence in $H_g^1(M)^\Gamma$.
\end{proposition}

\begin{proof}
 Let $(u_n)$ be a $(PS)_c^\Gamma$-sequence with $u_n\in\mathcal{N}^\Gamma$ for every $n\in\mathbb{N}$. Hereafter, $C>0$ will denote a positive constant, not necessarily the same one, and not depending on $n$ nor the functions in $H_g^m(M)^\Gamma$. By Lemma \ref{BS}, the sequence $(u_n)$ is bounded in the Hilbert space $H_g^m(M)^\Gamma$ and, passing to a subsequence, which we denote the same, there exists $u_0\in H_g^m(M)$ such that
 \begin{equation}\label{Equation:WeakConvergence}
 u_n\rightharpoonup u_0\qquad \text{ weakly in }\ H_g^m(M).
 \end{equation}
 
  Since $H_g^m(M)^\Gamma$ is a closed subspace of $H_g^m(M)$, it is weakly closed and $u_0\in H_g^m(M)^\Gamma$. If $r<2_{m,d}^*$, then the embedding $H_g^m(M)\hookrightarrow L^{s}$ is compact for $s=2,q,r$, implying that $H_g^m(M)^\Gamma\hookrightarrow L^{s}$ is also compact for the same values. When $r=2_{m,d}^*$ and the hypothesis \eqref{Hyp:Gamma} hold true, then $\alpha:=\min\{\dim\;\Gamma x\;:\; x\in M\}\geq 1$ and $r=2_{m,d}^*<2_{m,d-\alpha}^*$, yielding that the Sobolev embedding $H_g^m(M)^\Gamma\hookrightarrow L_g^s(M)$ is also compact for $s=2,q,r$ by Lemma \ref{Lemma:HebeyVaugon} and because $2<q<r$. Then in either case, after passing to another subsequence,
\begin{equation}\label{Equation:StrongConvergenceL^s}
u_n\to u_0 \qquad\text{strongly in } \  L_g^s(M), \text{ for }s=2,q,r.
\end{equation}
We divide the rest of the proof into two steps.

\smallskip
\noindent \textbf{Step 1.} \emph{$u_0\neq 0$.}  
\smallskip 

To see this, suppose, in order to get a contradiction, that $u_0=0$. Then, on the one hand, the strong convergence of $u_n\to u_0=0$ in $L_g^r(M)$ gives that
\begin{equation}\label{Equation:L^rConvergenceToZero}
\int_{M} \vert u_n\vert^r \; dV_g\to \int_M \vert u_0\vert^r\; dV_g=0\qquad \text{as }n\to\infty,
\end{equation}
and the strong convergence in $L^2_g(M)$ and in $L_g^q(M)$ together with Lemma \ref{Lemma:ConvergenceFuntion} yield that
\begin{equation}\label{Equation:ConvergenceFunctionToZero}
\int_M f_q(u_n)u_n \; dV_g \rightarrow \int_M f_q(u_0)u_0 \; dV_g=0 \qquad \text{as }n\to\infty
\end{equation}
in another subsequence. Moreover, as the sequence $(u_n)$ is bounded and as condition \eqref{Equation:p_subcritical} holds true, by Proposition \ref{Proposition:CriticalGN}, there exist $\Upsilon\in[\frac{1}{m},\frac{2}{m}]\subset(0,1)$ such that
\begin{equation}\label{Equation:L^pConvergenceToZero}
 \int_M \vert \nabla u_n\vert^p_g \; dV_g\leq C \Vert u_n\Vert_{\textbf{a}}^{p\Upsilon}\Vert u_n\Vert_2^{p(1-\Upsilon)}\leq C \Vert u_n\Vert_2^{p(1-\Upsilon)}\to 0
\end{equation}
because $u_n\to u_0=0$ in $L^2_g(M)$ and because $1-\Upsilon>0$ with our choice of $p$. We now show that, actually, $u_n\to u_0=0$ strongly in $H_g^m(M)^\Gamma$: indeed, as $u_n\in\mathcal{N}^\Gamma\subset\mathcal{N}$, the definition of the Nehari set given in \eqref{DefinitionNehari}, together with \eqref{Equation:L^rConvergenceToZero}, \eqref{Equation:ConvergenceFunctionToZero} and \eqref{Equation:L^pConvergenceToZero}, yield that
\[
\lim_{n\to\infty}\Vert u_n\Vert_{\textbf{a}}^2 = \mp \lim_{n\to\infty} \int_{M} \vert\nabla u_n\vert_g^p\; dV_g + \lim_{n\to\infty} \int_M f_q(u_n)u_n\; dV_g + \lim_{n\to\infty}\int_M \vert u_n\vert^r \; dV_g = 0 =\Vert u_0\Vert_{\textbf{a}}.
\]
Therefore, as weak convergence and convergence in norm in a Hilbert space imply strong convergence, the previous limit together with \eqref{Equation:WeakConvergence}, imply that $u_n\to u$ strongly in $H_g^m(M)$ and, as $\mathcal{N}^\Gamma$ is closed, we have that $u_0=0\in\mathcal{N}^\Gamma\subset\mathcal{N}$, which is a contradiction to definition \eqref{DefinitionNehari}. Thus, $u_0\neq 0$ as we claimed.

\smallskip
\noindent\textbf{Step 2.} \emph{Passing to a subsequence, $u_n\to u_0$ strongly in  $H_g^m(M)$ and $u_0\in\mathcal{N}^\Gamma$.}

\smallskip
First observe that Hölder inequality with $s$ and $\frac{s}{s-1}$, for $s=2,q$ and $r$, combined with  \eqref{Equation:StrongConvergenceL^s}, provide that
\begin{align*}
\left\vert \int_M \vert u_n\vert^{s-2}u_n(u_n-u_0)\; dV_g\right\vert & \leq \int_M \vert u_n\vert^{s-1}\vert u_n - u_0\vert \;dV_g\\
&\leq \left( \int_M \vert u_n\vert^s \; dV_g\right)^{\frac{s-1}{s}} \left( \int_M \vert u_n-u_0\vert^s \; dV_g\right)^{\frac{1}{s}}\\
&=\Vert u_n\Vert_{s}^{s-1}\Vert u_n-u_0\Vert_s\\
&\leq C \Vert u_n-u_0\Vert_s\to 0
\end{align*}
because $u_n\to u_0$ in $L^s_g(M)$ and, thus, $(u_n)$ is also bounded in $L^s_g(M)$. Therefore
\begin{equation}\label{Equation:L^2iso(1)}
\int_M \vert u_n\vert^{s-2}u_n(u_n-u_0)\; dV_g = o(1) \ \text{ as }n\to\infty,\qquad\text{ for }s=2,q,r.
\end{equation}
Analogously, by \eqref{crescimentoforigem} that
\begin{align*}
\left\vert \int_M f_q(u_n)(u_n-u_0) \; dV_g \right\vert & \leq \int_M \vert f_q(u_n) \vert \vert u_n-u_0\vert \; dV_g\\
& \leq C \int_M \vert u_n\vert \vert u_n-u_0\vert  \;dV_g + C\int_M\vert u_n\vert^{q-1}\vert u_n - u_0\vert \;dV_g\\
&\to 0,
\end{align*}
yielding that 
\begin{equation}\label{Equation:fiso(1)}
\int_M f_q(u_n)(u_n-u_0) \;dV_g = o(1) \ \text{ as }n\to\infty,
\end{equation}
as well. 

Now, by the Cauchy-Schwarz inequality for the metric $g$, together with Hölder inequality for $p$ and $\frac{p}{p-1}$, Proposition \ref{Proposition:CriticalGN}, the Sobolev embedding $H_g^m(M)\hookrightarrow L^2_g(M)$, and the strong convergence $u_n\to u_0$ in $L_g^2(M)$ given by \eqref{Equation:StrongConvergenceL^s}, we have that
\begin{align*}
&\left\vert\int_M \vert \nabla u_n\vert^{p-2}_g \langle \nabla u_n,\nabla (u_n-u_0)\rangle_g \; dV_g\right\vert \leq \int_M \vert \nabla u_n\vert^{p-2}_g  \vert\langle \nabla u_n,\nabla (u_n-u_0)\rangle_g\vert \; dV_g\\
&\leq  \int_M \vert \nabla u_n\vert^{p-1}_g  \vert \nabla (u_n-u_0)\vert_g \; dV_g
\leq \left(\int_M \vert \nabla u_n\vert^{p}_g \; dV_g\right)^{\frac{p-1}{p}} \left(\int_M \vert \nabla (u_n-u_0)\vert_g^p\; dV_g\right)^{\frac{1}{p}}\\
&\leq C \Vert u_n\Vert_{\textbf{a}}^{\Upsilon(p-1)}\Vert u_n\Vert_2^{(1-\Upsilon)(p-1)} \Vert u_n - u_0\Vert_{\textbf{a}}^{\Upsilon}\Vert u_n-u_0\Vert_2^{(1-\Upsilon)}\\
&\leq C \Vert u_n\Vert_{\textbf{a}}^{\Upsilon(p-1)}\Vert u_n\Vert_{\textbf{a}}^{(1-\Upsilon)(p-1)} \left(\Vert u_n\Vert^\Upsilon_{\textbf{a}} + \Vert u_0\Vert_{\textbf{a}}^{\Upsilon}\right)\Vert u_n-u_0\Vert_2^{(1-\Upsilon)}\\
&\leq C \Vert u_n-u_0\Vert_2^{(1-\Upsilon)} \to 0
\end{align*}
where $\Upsilon\in(0,1)$ because $p$ satisfies the condition \eqref{Equation:p_subcritical}, and where we used that the sequence $(u_n)$ is bounded in $H_g^m(M)$. Therefore,
\begin{equation}\label{Equation:DerivativeIso(1)}
\int_M \vert \nabla u_n\vert^{p-2}_g \langle \nabla u_n,\nabla (u_n-u_0)\rangle_g \; dV_g = o(1) \ \text{ as }n\to\infty.
\end{equation}
Moreover, as $u_n\rightharpoonup u_0$ weakly in $H_g^m(M)$, then
\begin{equation}\label{Equation:NormIso(1)}
\langle u_n, u_0\rangle_{\textbf{a}}=\Vert u_0\Vert^2_{\textbf{a}} + o(1)\quad \text{ as }n\to\infty.
\end{equation}
Finally, since $(u_n)$ is a $(PS)_c^\Gamma$ sequence for $I$ and it is bounded, we have that
\begin{align*}
\vert I'(u_n)(u_n-u_0)\vert&
\leq \Vert I'(u_n)\Vert_{(H_g^m(M))'}\Vert u_n - u_0\Vert_{\textbf{a}}\\
&\leq \Vert I'(u_n)\Vert_{(H_g^m(M))'}(\Vert u_n \Vert_{\textbf{a}} + \Vert u_0\Vert_{\textbf{a}})\\
&\leq C \Vert I'(u_n)\Vert_{(H_g^m(M))'}\to 0
\end{align*}
as $n\to\infty$ and hence
\begin{equation}\label{Equation:FunctionalIso(1)}
I'(u_n)(u_n-u_0) = o(1),\quad\text{ as }n\to\infty.
\end{equation}
Therefore, substituting \eqref{Equation:FunctionalIso(1)}, \eqref{Equation:NormIso(1)}, \eqref{Equation:fiso(1)} and \eqref{Equation:L^2iso(1)} for $s=r$ in the expression of $I'$ given in \eqref{Equation:DerivativeFunctional}, we obtain that
\begin{align*}
o(1)&=I'(u_n)(u_n-u_0)\\
& = \langle u_n, u_n-u_0 \rangle_{\textbf{a}} \pm \int_M \vert \nabla u_n\vert_g^{p-2}\langle \nabla u_n,\nabla(u_n-u_0)\rangle_g \; dV_g - \int_M f_q(u_n)(u_n-u_0) \; dV_g \\
&\hspace{2cm}- \int_M \vert u_n\vert^{r-2}u_n(u_n-u_0)\; dV_g\\
&= \Vert u_n\Vert_{\textbf{a}}^2 - \Vert u_0\Vert_{\textbf{a}}^2 + o(1)
\end{align*}
as $n\to\infty.$ From this, we infer that
\[
\Vert u_n\Vert_{\textbf{a}}\to\Vert u_0\Vert_{\textbf{a}}\quad \text{ as }n\to\infty.
\]
Again, as also $u_n\rightharpoonup u_0$  weakly in $H_g^m(M)$, we deduce that $u_n\to u_0$ strongly in $H_g^m(M)$ in a subsequence, and since $\mathcal{N}^\Gamma$ is  closed in $H_g^m(M)$, we have that $u_0\in\mathcal{N}^\Gamma$. This proves the proposition.
\end{proof}

\section{Critical Point Theory and proof of Theorem \ref{Theorem:Main}}\label{Section:CriticalPointTheory}

In this section, we maintain the same assumptions on the group $\Gamma$ and on the function $a_0\in C^\infty(M)$ as in Section \ref{Section:SymmetricSetting}. In order to handle the existence of critical points on the Nehari set $\mathcal{N}^\Gamma$, we follow the approach given in \cite{SzulkinWeth2010}. 
Recall the definition of the projection onto $\mathcal{N}$ given in Lemma \ref{Lemma:ProjectionOdd}. Then, its restriction to $H_g^m(M)^\Gamma$ induces a continuous and odd map
\[
\widehat{H}: H_g^m(M)^\Gamma\smallsetminus\{0\}\to \mathcal{N}^\Gamma,\quad \widehat{H}(u):= t_uu
\]
which also induces an odd homeomorphism $H$ between $\mathcal{N}^\Gamma$ and 
\[
S_1^\Gamma(0):= S_1(0)\cap H_g^m(M)^\Gamma.
\]
Consider the functionals $\widehat{\Psi}:H_g^m(M)^\Gamma\smallsetminus\{0\}\to \mathbb{R}$ and $\Psi:S_1^\Gamma(0)\to\mathbb{R}$ defined by
\[
\widehat{\Psi}(w):=I(\widehat{H}(w))= I(t_w w),\qquad \Psi:=\widehat{\Psi}\big\vert_{S_1^\Gamma(0)}.
\]
Observe that 
\[
\Psi(w) = I(t_w w),\quad \forall w\in S_1^\Gamma(0),\quad\text{and}\quad \Psi(H^{-1}(u))=\Psi\left(\frac{u}{\Vert u\Vert_{\textbf{a}}}\right) = I(u)\quad\forall\; u\in\mathcal{N}^\Gamma,
\]
by Corollary \ref{Corollary:Homeomorphism}.
As the regularity of $f$ does not assure that $\mathcal{N}^\Gamma$ is a $C^1$-Hilbert submanifold of $H_g^m(M)^\Gamma$, we will use the fact that there is a one-to-one correspondence between the critical points of $\Psi$ on $S_1^\Gamma(0)$ and the critical points of $I$ in $H_g^m(M)^\Gamma$. For this to happen, the functional $I$ and the functional space $H_g^m(M)^\Gamma$ must satisfy certain conditions, introduced in \cite{SzulkinWeth2010}, which we will next state and verify.

\begin{lemma}\label{Lemma:PropertiesA_1-A_3}
The following conditions hold true 

\begin{enumerate}[label=($A_1$),ref=$(A_1)$]
		\item \label{A_1} There exists an strictly increasing and continuous function $\varphi:[0,\infty)\to[0,\infty)$ such that $\varphi(0)=0$, $\varphi(t)\to\infty$ as $t\to\infty$ and such that the function
        \[
u\mapsto\xi(u):=\int_{0}^{\Vert u\Vert_{\textbf{a}}}\varphi(t) \;dt\in C^1(H_g^m(M)^\Gamma,\mathbb{R})
        \]
        satisfies that $\xi'$ is bounded on bounded sets and $\xi'(w)w=1$ for every $w\in S_1^\Gamma(0)$.
\end{enumerate}

\begin{enumerate}[label=($A_2$),ref=$(A_2)$]
		\item \label{A_2} For each $w\in H_g^m(M)^\Gamma\smallsetminus\{0\}$, there exists $t_w\in\mathbb{R}$ such that if $\mu_w(t):= I(tw)$, then $\mu_w'(t)>0$ for $0<t<t_w$ and $\mu_w'(t)<0$ for $t>t_w$. 
\end{enumerate}

\begin{enumerate}[label=($A_3$),ref=$(A_3)$]
		\item \label{A_3} There exists $\delta>0$ such that $t_w>\delta$ for all $w\in S_1^\Gamma(0)$ and for each compact subset $\mathcal{W}\subset S_1^\Gamma(0)$ there exists a constant $C_{\mathcal{W}}$ such that $t_w\leq C_{\mathcal{W}}$ for all $w\in\mathcal{W}$. 
\end{enumerate}
\end{lemma}

\begin{proof}
As remarked in Section 3 of \cite{SzulkinWeth2010}, taking $\varphi(t)=t$, condition \ref{A_1} is trivially fulfilled. Indeed, in this case $\xi(u)=\frac{1}{2}\Vert u\Vert_{\textbf{a}}$ which is of class $C^1$ in $H_g^m(M)^\Gamma$ because the norm $\Vert \cdot\Vert_{\textbf{a}}$ is an equivalent norm to the standard norm in $H_g^m(M)$ by Proposition \ref{Proposition:EquivalentNorms}. Moreover, $\xi'(u)w=\langle u, w\rangle_{\textbf{a}}$, which implies that $\xi'(w)w=\Vert w\Vert_{\textbf{a}}^2$ is bounded on bounded sets of $H_g^m(M)^\Gamma$ and $\xi'(w)w=\Vert w\Vert_{\textbf{a}}^2=1$ if $w\in S_1^\Gamma(0)$, proving that condition \ref{A_1} is satisfied. 

Condition \ref{A_2} follows directly from Lemma \ref{Lemma:NehariNonEmpty}. In fact, in Lemma \ref{Lemma:NehariNonEmpty}, for any $w\in H_g^m(M)^\Gamma\smallsetminus\{0\}\subset H_g^m(M)\smallsetminus\{0\}$, we proved the existence of a unique $t_w\in(0,\infty)$ such that $t_w w\in\mathcal{N}$. As $w\in H_g^m(M)^\Gamma$, then also $v:=t_ww$ satisfies that $v\circ\gamma= t_{w\circ\gamma}w\circ\gamma=t_w w=v$ and $t_ww\in\mathcal{N}^\Gamma$. Moreover, we proved in the same lemma that $\mu_w(t):= I_w(t)=I(tw)$ has a unique critical point $t_w\in(0,\infty)$ which is a maximum, and as $I$ is of class $C^1$, then $\mu_w$ is of class $C^1$ and $\mu_w'(t)>0$ for every $t\in(0,t_w)$ and $\mu_w'(t)<0$ for every $t>t_w$, concluding that condition \ref{A_2} holds true.

Finally, we verify condition \ref{A_3}. To check the first assertion in \ref{A_3}, recall that in Lemma \ref{Lemma:NehariClosed} we proved the existence of $K_0>0$ such that $\Vert u\Vert_{\textbf{a}}\geq K_0$, for every $u\in\mathcal{N}^\Gamma\subset\mathcal{N}$. Hence, taking $\delta:= K_0/2$, for any $w\in S_1^\Gamma(0)$ we have that $t_w w\in\mathcal{N}^\Gamma$ and
\[
t_{w}=t_w\Vert w\Vert_{\textbf{a}} = \Vert t_w w\Vert_\textbf{a}\geq K_0>\delta,
\]
as we wanted to verify. Now, to verify the second assertion in \ref{A_3}, in Lemma \ref{Lemma:ProjectionOdd}, we proved that the map $u\mapsto t_u\in(0,\infty)$ is continuous in $H_g^m(M)^\Gamma\smallsetminus\{0\}\subset H_g^m(M)\smallsetminus\{0\}$. Therefore, for any compact subset $\mathcal{W}\subset S^\Gamma_1(0)$, this map attains its maximum $C_\mathcal{W}>0$, yielding that $t_w\leq C_{\mathcal{W}}$ for any $w\in \mathcal{W}$, and the lemma follows. 
\end{proof}

It follows from condition \ref{A_1} that $S_1^\Gamma(0)$ is a codimension one $C^1$-Hilbert submanifold of $H_g^m(M)^\Gamma$ and that the tangent space of $S_1^\Gamma(0)$ at $w\in S_1^\Gamma(0)$ is given by
\[
T_w(S_1^\Gamma(0))= \{ u\in H_g^m(M)^\Gamma \;:\; \xi'(w)u=0 \}.
\]
As conditions \ref{A_1}-\ref{A_3} are fulfilled, it follows from Proposition 9 in \cite{SzulkinWeth2010} that $\widehat{\Psi}\in C^1(H_g^m(M)^\Gamma\smallsetminus\{0\},\mathbb{R})$ and its derivative is given by
\[
\widehat{\Psi}'(w) u = \frac{\Vert t_w w\Vert_{\textbf{a}}}{ \Vert w\Vert_{\textbf{a}} } I'(t_ww)u, \qquad\forall\; u,w\in H_g^m(M)^\Gamma, w\neq 0.
\]
Moreover, by Corollary 10 in \cite{SzulkinWeth2010}, $\Psi\in C^1(S_1^\Gamma(0),\mathbb{R})$ and
\[
\Psi'(w)u = \Vert t_w w\Vert_{\textbf{a}} I'(t_ww)u,\qquad \forall \; u\in T_w(S_1^\Gamma(0)).
\]
For any $w\in S_1^\Gamma(0)$, the norm of $\Psi'(w)$ in the cotangent space $T_w^\ast(S_1^\Gamma(0))$ is given by
\[
\Vert \Psi'(w) \Vert_{\star}:= \sup_{v\in T_w(S_1^\Gamma(0)),\\ v\neq 0} \frac{\vert \Psi'(w)v\vert}{\Vert v \Vert_{\textbf{a}}}.
\]
Given $c\in\mathbb{R}$, a sequence $(w_n)$ in $S_1^\Gamma(0)$ is a $(PS)_c^\Gamma$-sequence for $\Psi$ if $\Psi(w_n)\to c$ and $\Vert \Psi'(w_n)\Vert_{\star}\to 0$. The functional $\Psi$ is said to satisfy the $(PS)_c^\Gamma$-condition in $S_1^\Gamma(0)$ if every such sequence has a convergent subsequence. We are ready to establish the desired correspondence between the critical points of $\Psi$ in $S_1^\Gamma(0)$ and the critical point of $I$ in $H_g^m(M)^\Gamma$, which is a consequence of \ref{A_1}-\ref{A_3}.

\begin{proposition}[Corollary 10 in \cite{SzulkinWeth2010}]\label{Corollary:WethSzulkin}

\begin{enumerate}[label=(\roman*)]
\item Given $c\in\mathbb{R}$, if $(w_n)$ is a $(PS)_c^\Gamma$-sequence for $\Psi$ on $S_1^\Gamma(0)$, then $(t_{w_n}w_n)$ is a $(PS)_c^\Gamma$-sequence for $I$ with $t_{w_n}w_n\in\mathcal{N}^\Gamma$. If $(u_n)$ is a bounded $(PS)_c^\Gamma$-sequence for $I$ with $u_n\in\mathcal{N}^\Gamma$ for every $n\in\mathbb{N}$, then $(H^{-1}(u_n))$ is a $(PS)_c^\Gamma$-sequence for $\Psi$.
\item $w$ is a critical point of $\Psi$ on $S_1^\Gamma(0)$ if and only if $t_w w$ is a nontrivial critical point of $I$ in $H_g^m(M)^\Gamma$. Moreover, the corresponding values of $\Psi$ and $I$ coincide and
\[
\inf_{S_1^\Gamma(0)}\Psi = \inf_{\mathcal{N}^\Gamma} I.
\]

\end{enumerate}
\end{proposition}

We have everything to prove our main result.

\begin{proof}[Proof of Theorem \ref{Theorem:Main}]

We adapt the proof of Theorem 12 in \cite{SzulkinWeth2010} to our situation. By Lemma \ref{Lemma:PropertiesA_1-A_3}, the Hilbert space $H_g^m(M)^\Gamma$ and the functional $I$ satisfies properties \ref{A_1} - \ref{A_3} above; therefore, if $(w_n)$ is a $(PS)_c^\Gamma$ sequence for the functional $\Psi$ in $S_1^\Gamma(0)$, then $u_n:= t_{w_n}w_n$ is a $(PS)_c^\Gamma$ sequence for $I$ with $u_n\in\mathcal{N}^\Gamma$.

We prove next that $\Psi$ satisfies the $(PS)_c^\Gamma$-condition in $S_1^\Gamma(0)$ for every $c\in\mathbb{R}$. Indeed, consider $c\in\mathbb{R}$ and a $(PS)_c^\Gamma$-sequence $(w_n)$ in $S_1^\Gamma(0)$. By the above remark, $(u_n:=t_{w_n}w_n)$ is a $(PS)_c^\Gamma$-sequence for $I$ with $u_n\in\mathcal{N}^\Gamma$, and by Proposition \ref{Propo:Palais-Smale}, there exists $u_0\in\mathcal{N}^\Gamma$ such that $u_n\to u_0$ in $H_g^m(M)^\Gamma$ in a subsequence. As the function $H^{-1}:\mathcal{N}\to S_1(0)$ given in Corollary \ref{Corollary:Homeomorphism} is continuous, we have that $w_n=H^{-1}(u_n)\to w_0:=H^{-1}(u_0)$ and by continuity of $\Psi$, $\Psi(w_n)\to \Psi(w_0)=c$ in a subsequence, proving that $\Psi$ satisfies the $(PS)_c^\Gamma$-condition in $S_1^\Gamma(0)$.

Now we see the existence of a least energy $\Gamma$-invariant solution to \eqref{Problem:Main}. Let $(w_n)$ be a minimizing sequence for $\Psi$ in $S_1^\Gamma(0)$, i.e., $\Psi(w_n)\to c_0^\Gamma$. By Ekeland's variational Principle \cite{Ekeland1974,MawhinWillemBook}, we may assume that $\Psi'(w_n)\to 0$, so that $(w_n)$ is a $(PS)_{c_0^\Gamma}^\Gamma$-sequence in $S_1^\Gamma(0)$. Since $\Psi$ satisfies the $(PS)_{c_0^\Gamma}^\Gamma$-condition in $S_1^\Gamma(0)$, after passing to a subsequence we have that $w_n\to w_0\in S_1^\Gamma(0)$ in $H^m_g(M)^\Gamma$,  and $w$ is a minimizer for $\Psi$ in $S_1^\Gamma(0)$. By Proposition \ref{Corollary:WethSzulkin}, $u_0:=t_{w_0}w_0\in\mathcal{N}^\Gamma$ is a minimizer for $I$ in $\mathcal{N}^\Gamma$ and a nontrivial critical point of $I$ in $H_g^m(M)^\Gamma$. As $I$ is $\Gamma$-invariant by Lemma \ref{Lemma:FunctionalGammaInvariant}, the Principle of Symmetric Criticality \cite{Palais1979} yields that $u_0$ is a $\Gamma$-invariant solution to the problem \eqref{Problem:Main} with least energy $I(u_0)=c_0^\Gamma$. Notice that when $\Gamma=\{Id\},$ we have that $H_g^m(M)^\Gamma=H_g^m(M)$, $\mathcal{N}^\Gamma=\mathcal{N}$ and $c_0^\Gamma=c_0$, yielding that $u_0$ is a ground state in this case.

Finally, condition \ref{f_4} implies that the functional $I$ is even and since the function $H$ is odd by Lemma \ref{Lemma:ProjectionOdd}, we have that $\Psi=I\circ H$ is also even. Moreover, since $\Gamma$ satisfies \eqref{ActionNotTransitive}, as we remarked at the beginning of Section \ref{Section:SymmetricSetting},  $H_g^m(M)^\Gamma$ is an infinite dimensional Hilbert space. Furthermore, as $\Psi\in C^1(S_1^\Gamma(0))$ and as $\inf_{S_1^\Gamma(0)}\Psi=\inf_{\mathcal{N}^\Gamma} I=c_0^\Gamma>0$, then the functional $\Psi$ is bounded from below on $S_1^\Gamma(0)$ and as it satisfies the $(PS)_c^\Gamma$-condition in $S_1^\Gamma(0)$ for every $c\in\mathbb{R}$, then Theorem 2 in \cite{SzulkinWeth2010} applies (see also \cite{RabinowitzBook}), and$\Psi$ has infinitely many pairs of nontrivial critical points in $H_g^1(M)^\Gamma$ with increasing values. Again, by Proposition \ref{Corollary:WethSzulkin}, $I$ has infinitely many pairs of nontrivial critical points in $H_g^m(M)^\Gamma$ with increasing energy and by the Principle of Symmetric Criticality \cite{Palais1979}, problem \eqref{Problem:Main} admits an infinite number of $\Gamma$-invariant nontrivial solutions.

\end{proof}

\begin{proof}[Proof of Corollary \ref{Corollary:Kirschoff-BoussinesqJGMS}] Since the scalar curvature in an Einstein manifold is constant (see \cite{BesseBook}), as a direct consequence of Theorem 12 in \cite{Gover06}, we have that the operator $\mathscr{P}_g^m$ has the form
\[
\mathscr{P}_g^m=\prod_{\ell=1}^m (-\Delta_g + c_\ell \text{Sc}_g)
\]
where $c_\ell:=(d+2\ell-2)(d-2\ell)/(4d(d-1)>0$ because $d>2m$, and where $\text{Sc}_g>0$ denotes the scalar curvature of $(M,g)$, and, therefore, it has the form \eqref{Def:Operator} with $a_i>0$ for every $i$. Then the corollary follows directly from Theorem \ref{Theorem:Main}.
\end{proof}


\appendix

\section{Gagliardo-Nirenber interpolation inequality and equivalent norms in the space $H_g^m(M)$} \label{Appendix:ProofGagliardo-Nirenberg}

In this appendix, we prove Theorem \ref{Theorem:Gagliardo-Nirenberg} and use it to give several equivalent norms in the Hilbert space $H_g^m(M)$

\subsection{Proof Gagliardo-Nirenberg inequality}

Consider the space $H_g^m(M)$ with the standard norm $\Vert\cdot\Vert_{H_g^m(M)}$ given in \eqref{Equation:UsualSobolevNorm}. In order to prove Theorem \ref{Theorem:Gagliardo-Nirenberg}, we use the following well known result.

\begin{theorem}(\cite[Theorem 3.70]{AubinBook})\label{Theorem:Gagliardo-NirenbergZeroMean}
Let $(M,g)$ be a closed Riemannian manifold of dimension $d$ Let $j, m$ be integers satisfying that $0 \leq j < m$, $1\leq \kappa_1,\kappa_2,\kappa_3\leq\infty$ be real numbers, and $\Upsilon\in[0,1]$ satisfying that
\begin{align*}
    \frac{1}{\kappa_1}=\frac{j}{d}+ \left(\frac{1}{\kappa_3}-\frac{m}{d}\right)\Upsilon + \frac{(1-\Upsilon)}{\kappa_2} \quad \mbox{and} \quad \frac{j}{m}\leq \Upsilon \leq 1.
\end{align*}
 Then there exists a constant $C=C(d,m,j,\kappa_2,\kappa_3,\Upsilon,M)>0$ such that for all $u \in C^\infty(M)$ with
 \[
 \int_M u \, d V_g=0,
 \]
 the following inequality holds true
\begin{align*}
  \Big( \int_{M}  |\nabla^{j}_{g} u |^{\kappa_1} \,dV_g    \Big)^{\frac{1}{\kappa_1}}  \leq C   \Big(\int_{M}  |\nabla^{m}u |^{\kappa_3} \, dV_g     \Big)^{\frac{\Upsilon}{\kappa_3}} \Big( \int_{M}  |u |^{\kappa_2} \, dV_g  \Big)^{\frac{1-\Upsilon }{\kappa_2}} .
\end{align*}
\end{theorem}

We now show that the inequality in Theorem \ref{Theorem:Gagliardo-Nirenberg} is valid for any function $u\in H_g^m(M).$ To see this, first define
\[
V:=\left\{u\in H_g^m(M) \;:\; \int_M u \;dV_g=0\right\}
\]
be the subspace of $H_g^m(M)$ consisting of functions with zero mean and define
\[
W:=\text{span}(\{1\})
\]
be the subspace of $H_g^m(M)$ generated by the constant functions. Since $M$ is compact, we have the embeddings $H_g^m(M)\hookrightarrow L^2_g(M)\hookrightarrow L_g^1(M)$, and therefore, the space $V$ is well defined, and the function $T:H_g^m(M)\rightarrow\mathbb{R}$ given by $T(u)=\int_M u\; dV_g$ is linear and continuous. This implies that the space $V=T^{-1}(0)$ is closed in $H_g^m(M).$ Since $\dim W=1$, then also $W$ is closed in $H_g^m(M)$. We have the following decomposition of $H_g^m(M)$ in terms of these spaces.

\begin{lemma}
\[
H_g^m(M) = V\oplus W\qquad\quad \text{and}\qquad V^\perp= W.
\]
\end{lemma}

\begin{proof}
For any $u\in H_g^m(M)$, we can write
\begin{equation}\label{Equation:OrtogonalDecomposition}
u = (u-\bar{u}) + \bar{u}, \quad \bar{u}=\frac{1}{\text{Vol}_g(M)}\int_M u\; dV_g
\end{equation}
where $v=u-\bar{u}\in V$ and $\bar{u}\in W$. Clearly $V\cap W=\{0\}$ and, hence, $H_g^m(M)=V\oplus W$. Next, we prove that $V^{\perp}=W.$ Indeed, given a constant function $u\equiv c\in W$, we have that $\vert\nabla^jc\vert_g=0$ for each $j=1,\ldots,m$ and, therefore, for any $v\in V$
\[
\langle u, v\rangle_{H_g^1(M)} = \int_M cv \; dV_g +   \sum_{j=1}^m \int_M\langle \nabla^jc,\nabla^j v\rangle_g\;dV_g = c\int_M v\; dV_g = 0,
\]
where we infer that $W\subset V^\perp$. Now, as $V$ is closed in the Hilbert space $H_g^m(M)$, we have that $H_g^m(M) =V\oplus V^\perp$; hence 
\[
H_g^m(M) = V\oplus W\subset V\oplus V^{\perp}= H_g^{m}(M),
\]
which proves that $V\oplus W = V\oplus V^\perp$. From this, we obtain that $W=V^\perp$: indeed, if $W\subsetneq V^\perp$, then there exists $u\in V^\perp$ such that $u\notin W$. As $H_g^m(M)=V\oplus M$, there exists a unique $v\in V$ and a unique constant function $w\equiv c\in W$ such that $u=v+c$. Since $u\notin W$, we have that $v\neq0$, and as $W\subset V^{\perp}$,
\[
0=\langle u,v\rangle_{H_g^m(M)} = \langle v+c,v\rangle_{H_g^m(M)} = \Vert v\Vert^2_{H_g^m(M)} + \langle c,v\rangle_{H_g^m(M)} = \Vert v\Vert^2_{H_g^m(M)}\neq 0
\]
which is a contradiction and $W=V^\perp$, as we wanted to prove.
\end{proof}

We use this decomposition to prove Theorem \ref{Theorem:Gagliardo-Nirenberg}.

\begin{proof}[Proof of Theorem \ref{Theorem:Gagliardo-Nirenberg}]
Let $u\in C^\infty(M)\subset H_g^m(M)$ and consider the decomposition $u=v+w$ with $v=u-\bar{u}\in V$ and $w=\bar{u}$ a constant function given in \eqref{Equation:OrtogonalDecomposition}. 
Given any $s\geq 1$, since the function $\xi(t)=\vert t\vert^s$ is convex and $(M,g)$ is compact, by Jensen's inequality we have for any $u\in L^s(M)$ that
\begin{equation}\label{Inequality:Jensen}
\left\vert \frac{1}{\text{Vol}_g(M)} \int_{M} u \; dV_g \right\vert^s = \xi\left(\frac{1}{\text{Vol}_g(M)}\int_M u \;dV_g\right)\leq \frac{1}{\text{Vol}_g(M)}\int_M \xi(u)\; dV_g = \frac{1}{\text{Vol}_g(M)}\int_M \vert u\vert^s\; dV_g. 
\end{equation}

Moreover, for $1\leq j\leq m$, it follows that $\vert \nabla^j u\vert_g = \vert\nabla^j (v+w)\vert_g = \vert \nabla^j v\vert_g$. Hence, for any $1\leq\kappa_1,\kappa_2,\kappa_3<\infty$ and $\Upsilon\in[0,1]$ satisfying the hypotheses in Theorem  \ref{Theorem:Gagliardo-Nirenberg}, the hypotheses in Theorem \ref{Theorem:Gagliardo-NirenbergZeroMean} also hold for any $v\in V$ and by  \eqref{Inequality:Jensen} with $s=\kappa_2$, we obtain that
\begin{align*}
 &\Big( \int_{M}  |\nabla^{j}_{g} u |^{\kappa_1} \,dV_g \Big)^{\frac{1}{\kappa_1}}  
 = \Big( \int_{M}  |\nabla^{j}_{g} v |^{\kappa_1} \,dV_g \Big)^{\frac{1}{\kappa_1}}  \\
 &\leq C   \Big(\int_{M}  |\nabla^{m}v |^{\kappa_3} \, dV_g     \Big)^{\frac{\Upsilon}{\kappa_3}} \Big( \int_{M}  |v |^{\kappa_2} \, dV_g  \Big)^{\frac{1-\Upsilon }{\kappa_2}} \\
 &\leq C   \left[ \left(\int_{M}  |\nabla^{m}v |^{\kappa_3} \, dV_g\right)^{1/\kappa_3} +  \left(\int_{M}  \vert u\vert^{\kappa_3} \, dV_g\right)^{1/\kappa_3}    \right]^{\Upsilon} \Big( \int_{M}  |v |^{\kappa_2} \, dV_g  \Big)^{\frac{1-\Upsilon }{\kappa_2}}\\
 &=C   \left[ \left(\int_{M}  |\nabla^{m}u |^{\kappa_3} \, dV_g\right)^{1/\kappa_3} +  \left(\int_{M}  \vert u\vert^{\kappa_3} \, dV_g\right)^{1/\kappa_3}    \right]^{\Upsilon} \Big( \int_{M}  |u-\bar{u} |^{\kappa_2} \, dV_g  \Big)^{\frac{1-\Upsilon }{\kappa_2}}\\
 &\leq 2^{(1-\Upsilon)} C  \left[ \left(\int_{M}  |\nabla^{m}u |^{\kappa_3} \, dV_g\right)^{1/\kappa_3} +  \left(\int_{M}  \vert u\vert^{\kappa_3} \, dV_g\right)^{1/\kappa_3}    \right]^{\Upsilon} \Big( \int_{M}  \vert u\vert^{\kappa_2} + \vert \bar{u} \vert^{\kappa_2} \, dV_g  \Big)^{\frac{1-\Upsilon }{\kappa_2}}\\
 &= 2^{(1-\Upsilon)} C  \left[ \left(\int_{M}  |\nabla^{m}u |^{\kappa_3} \, dV_g\right)^{1/\kappa_3} +  \left(\int_{M}  \vert u\vert^{\kappa_3} \, dV_g\right)^{1/\kappa_3}    \right]^{\Upsilon} \Big( \int_{M}  \vert u\vert^{\kappa_2} + \left\vert \frac{1}{\text{Vol}_g(M)} \int_M u \; dV_g \right\vert^{\kappa_2} \, dV_g  \Big)^{\frac{1-\Upsilon }{\kappa_2}}\\
 &\leq 2^{(1-\Upsilon)} C  \left[ \left(\int_{M}  |\nabla^{m}u |^{\kappa_3} \, dV_g\right)^{1/\kappa_3} +  \left(\int_{M}  \vert u\vert^{\kappa_3} \, dV_g\right)^{1/\kappa_3}    \right]^{\Upsilon} \Big( \int_{M}  \vert u\vert^{\kappa_2} +  \left(\frac{1}{\text{Vol}_g(M)} \int_M \vert u\vert^{\kappa_2} \; dV_g\right)  \, dV_g  \Big)^{\frac{1-\Upsilon }{\kappa_2}}\\
 &= 2^{(1-\Upsilon)} C  \left[ \left(\int_{M}  |\nabla^{m}u |^{\kappa_3} \, dV_g\right)^{1/\kappa_3} +  \left(\int_{M}  \vert u\vert^{\kappa_3} \, dV_g\right)^{1/\kappa_3}    \right]^{\Upsilon} \left( 2\int_{M}  \vert u\vert^{\kappa_2} dV_g\right)^{\frac{1-\Upsilon }{\kappa_2}}\\
 &= 2^{(1-\Upsilon)(1+\frac{1}{\kappa_2})} C  \left[ \left(\int_{M}  |\nabla^{m}u |^{\kappa_3} \, dV_g\right)^{1/\kappa_3} +  \left(\int_{M}  \vert u\vert^{\kappa_3} \, dV_g\right)^{1/\kappa_3}    \right]^{\Upsilon} \left( \int_{M}  \vert u\vert^{\kappa_2} dV_g\right)^{\frac{1-\Upsilon }{\kappa_2}}
\end{align*}
and the constant $2^{(1-\Upsilon)(1+\frac{1}{\kappa_2})}C$ depends only on $d,m,j,\kappa_1,\kappa_2,\kappa_3,\Upsilon$ and $M$ because $C=C(d,m,j,\kappa_1,\kappa_3,\Upsilon,M)$. The inequality follows for any $u\in H_g^m(M)$ by a density argument.
\end{proof}

\subsection{Equivalent norms in $H_g^m(M)$ and proof of Proposition \ref{Propo:EquivalentNorm}} \label{Subsection:Appendix:ProofEquivalentNorms}

Let $m\geq 2$ be a positive integer and consider a closed Riemannian manifold $(M,g)$ of dimension $d> 2m$. There are two main ways to define the higher order Sobolev space $H_g^m(M)$. The first one, as the completion of $C^\infty(M)$ with respect to the norm \eqref{Equation:UsualSobolevNorm}.
The second norm, which is mainly used in the variational framework of problems involving higher order laplacians, is defined as the completion of $C^\infty(M)$ with respect to the norm 
\begin{equation}\label{Equation:SObolevNormHigherOrderLaplacian}
\Vert u\Vert_{\Delta}:= \left( \sum_{\substack{i=0\\ i\  even}}^m  \int_M  \vert \Delta_g^{i/2}  u\vert^2 \; dV_g
+ \sum_{\substack{i=0\\ i\  odd}}^m \int_M  \vert \nabla\Delta_g^{(i-1)/2}u\vert_g^2 \; dV_g\right)^{1/2},\quad u\in C^\infty(M).
\end{equation}
which is induced by the inner product
\begin{equation}\label{Equation:InnerProductHigherOrderLaplacians}
\langle u,v\rangle_{\Delta}:= 
\sum_{\substack{i=0\\ i\  even}}^m  \int_M \Delta_g^{i/2} u\Delta_g^{i/2} v \; dV_g
+ \sum_{\substack{i=0\\ i\  odd}}^m \int_M \langle \nabla\Delta_g^{(i-1)/2}u, \nabla\Delta_g^{(i-1)/2} v\rangle_g \; dV_g,\quad u,v\in \mathcal{C}^\infty(M),
\end{equation}
In this section, we prove that both norms are equivalent, and also that they are equivalent to the norm $\Vert\cdot\Vert_{\text{a}}$, given in \eqref{Equation:EquivalentNorm}, for any $\textbf{a}\in C_+^\infty(M)\times[0,\infty)^{m-1}\times(0,\infty)$.
Remark that the norm $\Vert\cdot\Vert_{\Delta}$ is just the norm $\Vert\cdot\Vert_{\textbf{a}}$ with $\textbf{a}=(1,\ldots,1)$. 

We can give two more norms that are interesting because they only use the higher order derivatives and the zero-th order term. The first one is just $\Vert\cdot\Vert_{\textbf{a}}$ with $\textbf{a}=(1,0,\ldots,0,1)$, which is given explicitly by
\[
\Vert u\Vert_{\ast}:=\Vert u\Vert_{(1,0,\ldots,0,1)}:=\begin{cases}
\left( \int_M (\Delta^{m/2} u)^2 + u^2 \;dV_g\right)^{1/2} & \text{ if }m \text{ is even, }\\
\left( \int_M \vert\nabla\Delta^{(m-1)/2} u\vert_g^2 + u^2 \;dV_g\right)^{1/2} & \text{ if }m \text{ is even. }
\end{cases}
\]
The other norm is
\[
\Vert u\Vert_{\ast\ast}:=\left(\int_M  \left[ \vert \nabla^m u\vert_g^2 + u^2\; dV_g   \right] \right)^{1/2}, \qquad u\in H_g^m(M)
\]
which is induced by the inner product
\[
\langle u,v \rangle_{\ast\ast}:=\int_M \left[\langle \nabla^m u,\nabla^m v \rangle_{g} + uv\right] \; dV_g
\]

The main result in this section is the following:

\begin{proposition}\label{Propo:EquivalentNorm} 
For any $\textbf{a}=(a_0,a_1,\ldots,a_{m-1},a_m)\in C_+^\infty(M)\times [0,\infty)^{m-1}\times(0,\infty)$, the norm $\Vert \cdot\Vert_{\textbf{a}}$ is equivalent to the norms $\Vert\cdot\Vert_{H_g^m(M)}$, $\Vert\cdot\Vert_{\Delta}$, $\Vert\cdot\Vert_\ast$, and $\Vert\cdot\Vert_{\ast\ast}$.
\end{proposition}

As an application of the Gagliardo-Nirenberg interpolation inequality (Theorem \ref{Theorem:Gagliardo-Nirenberg}), we begin by showing that the norm
$\Vert\cdot\Vert_{H_g^m(M)}$ is equivalent to $\Vert\cdot\Vert_{\ast\ast}$.

\begin{lemma}\label{Lemma:EstimatesDerivatives}
For $m,d\in\mathbb{N}$ such that $2m<d$, there exists a constant $C>0$ such that
\[
\left(\int_M \vert\nabla^j u\vert^2_g \; dV_g\right)^{1/2}\leq C \left( \int_M \left[\vert \nabla^m u\vert_g^2 + u^2\right]\; dV_g\right)^{1/2}
\]
for any $j=1,\ldots,m-1$ and any $u\in H_g^m(M)$. In particular, $\Vert\cdot\Vert_{\ast\ast}$ 
is equivalent to the norm $\Vert\cdot\Vert_{H_g^m(M)}$ in $H_g^m(M)$.
\end{lemma}

\begin{proof}
For any $j\in\{1,\ldots,m-1\}$, fix $\kappa_2=2,\Upsilon =1$ and let $\kappa_2\geq 1$ be arbitrary, and take $\kappa_3\in\mathbb{R}$ to be defined as
\[
\frac{1}{\kappa_3} := \frac{1}{2}-\frac{j}{d}+\frac{m}{d} 
\]
We now show that $\kappa_1,\kappa_2,\kappa_3,j,m$ and $\Upsilon$ satisfy the hypothesis in Theorem \ref{Theorem:Gagliardo-Nirenberg}. First, since $j\leq m-1$, we immediately have that
\[
\frac{j}{m}\leq\frac{m-1}{m}<\Upsilon =1,
\]
and, by our choice of the constants $\kappa_1,\kappa_3$ and $\Upsilon$, we have that
\[
\frac{1}{\kappa_1}=\frac{1}{2}= \frac{j}{d} - \frac{m}{d} + \frac{1}{\kappa_3} = \frac{j}{d}+ \left(\frac{1}{\kappa_3}-\frac{m}{d}\right)\Upsilon + \frac{(1-\Upsilon)}{\kappa_2}.
\]
We verify that $\kappa_3\in(1,2)$. Indeed,  since $2m<d$, we have that
\begin{align*}
    \frac{1}{k_3}=  \frac{1}{2}+ \frac{m}{d}- \frac{j}{d} < \frac{1}{2}+ \frac{m}{d} = \frac{1}{2} + \frac{2m}{2d}<  \frac{1}{2}+ \frac{1}{2}=1,
\end{align*}
implying that $1 <  k_3 $. Furthermore, as $j\leq m-1$, we have that
\begin{align*}
    \frac{1}{k_3}=  \frac{1}{2}+ \frac{m}{d}- \frac{j}{d} \geq \frac{1}{2}+ \frac{m}{d}-\frac{m-1}{d} = \frac{1}{2}+ \frac{1}{d} >\frac{1}{2}
\end{align*}
yielding that $k_3 < 2$. Hence, we can apply the Gagliardo-Nirenberg interpolation inequality in Theorem \ref{Theorem:Gagliardo-Nirenberg} together with the embedding $L_g^2(M)\hookrightarrow L_g^{\kappa_3}(M)$ to obtain a positive constant $C_j>0$ such that

\[
  \Big( \int_{M}  |\nabla^{j} u |_g^{2} \,dV_g    \Big)^{1/2}  \leq C_j \Big[  \Vert \vert \;\nabla^m u\vert_g \Vert_{\kappa_3} + \Vert u \Vert_{\kappa_3}    \Big]
  \leq C_j\text{Vol}_g(M)^{\frac{2-\kappa_3}{2\kappa_3}} \Big[  \Vert \vert \;\nabla^m u\vert_g \Vert_{2} + \Vert u \Vert_{2}    \Big],
\]
and therefore, for any $j=1,\ldots,m-1$,
\[
\int_M \vert \nabla^j u\vert_g^2 \; dV_g \leq C_j^2\text{Vol}_g(M)^{2\left(\frac{2-\kappa_3}{\kappa_3}\right)} \Big[  \Vert \vert \;\nabla^m u\vert_g \Vert_{2} + \Vert u \Vert_{2}    \Big]^2\leq D_j \int_M \left[ \vert \nabla^m u\vert_g^2 + u^2\right] \; dV_g
\]
where $D_j:=4 C_j^2\text{Vol}_g(M)^{2\left(\frac{2-\kappa_3}{\kappa_3}\right)}$.

Taking $C:=\max\{\sqrt{D_1},\ldots,\sqrt{D_{m-1}}\}$, we conclude the first part of the lemma. It is now straightforward to see that $\Vert\cdot\Vert_{\ast\ast}$ is a norm in $H_g^m(M)$ equivalent to $\Vert\cdot\Vert_{H_g^m(M)}.$
\end{proof}

Now, as a consequence of integration by parts, we state the equivalence between $\Vert\cdot\Vert_{H_g^m(M)}$ and $\Vert\cdot\Vert_{\Delta}$. The proof can be found in \cite{Salomonsen2001} and in \cite{SandovalVelazquez2026}, where a more general setting is considered.

\begin{lemma}\label{Lemma:EquivalenceUsualNorms}
The norms $\Vert\cdot\Vert_{H_g^m(M)}$ and $\Vert\cdot\Vert_\Delta$  are equivalent in $H_g^m(M).$
\end{lemma}

The last equivalence we will establish is between the norm $\Vert\cdot\Vert_{H_g^m(M)}$ and the norm $\Vert\cdot\Vert_{\ast}$. This is a direct consequence of a stronger result, known as G\aa rding's inequality.

\begin{lemma}[G\aa rding's inequality]\label{Lemma:GardingInequality} If $\mathcal{P}:C^\infty(M)\to C^\infty(M)$ is a strongly elliptic operator of order $2m$, then, there exists a constant $C>0$ such that
\[
\Vert u\Vert_{H_g^m(M)}^2 \leq C\left[ \int_M u\mathcal{P}u\; dV_g + \Vert u\Vert_2^2 \right],\qquad\forall\;u\in C^\infty(M).
\]
In particular, the norm  $\Vert\cdot\Vert_{H_g^m(M)}$ and $\Vert\cdot\Vert_{\ast}$ are equivalent in $H_g^m(M)$.
\end{lemma}

\begin{proof}
For a proof of G\aa rding's inequality on Riemannian manifolds, see, for instance, \cite[Section 5.11]{TaylorBookI}. Now, as the $m$-th power of the Laplace-Beltrami operator, $\mathcal{P}:=(-\Delta_g)^m$, is strongly elliptic, we have that
\[
\Vert u\Vert_{H_g^m(M)}^2\leq C \int_{M} \left[u(-\Delta_g^m)u + u^2 \right] \; dV_g = C \Vert u\Vert_{\ast}^2,\qquad \forall\; u\in C^\infty(M).
\]
Since $C^\infty(M)$ is dense in $H_g^m(M)$ and since the norms $\Vert\cdot\Vert_{\textbf{a}}$ can be extended to $H_g^m(M)$, then $\Vert u\Vert_{H_g^m(M)}\leq \sqrt{C}\Vert u\Vert_{\ast}$ for every $u\in H_g^m(M)$. In addition, by Lemma \ref{Lemma:EquivalenceUsualNorms}, there exists a constant $C'>0$ such that
\[
\Vert u\Vert_{\ast}\leq \Vert u\Vert_\Delta\leq C' \Vert u\Vert_{H_g^m(M)},\qquad\forall\; u\in H_g^m(M),
\]
concluding the equivalence of the norms $\Vert u\Vert_{\ast}$ and $\Vert u\Vert_{H_g^m(M)}$.
\end{proof}

\begin{proof}[Proof of Proposition \ref{Propo:EquivalentNorm}]
By Lemmas \ref{Lemma:EstimatesDerivatives}, \ref{Lemma:EquivalenceUsualNorms} and \ref{Lemma:GardingInequality}, it suffices to show that $\Vert \cdot\Vert_\ast$ and $\Vert\cdot\Vert_{\textbf{a}}$ are equivalent in $H_g^m(M)$ for every $\textbf{a}=(a_0,a_1,\ldots,a_{m-1},a_m)\in C_+^\infty(M)\times [0,\infty)^{m-1}\times(0,\infty)$. Let 
\[
0<\underline{a}_0:=\min_{x\in M}a_0(x)\leq a_0(x)\leq \max_{x\in M}a_0(x)=:\overline{a}_0.
\]
Then, on the one hand, as $a_m>0$, we have for every $u\in C^\infty(M)$ that
\begin{align*}
\Vert u\Vert_\ast^2 &= \int_M \left[u(-\Delta_g^m)u + u^2 \right]\; dV_g 
\leq \max\{\frac{1}{a_m},1,\underline{a}_0\}\int_{M} u\left[ a_m(-\Delta_g)^m + a_0(-\Delta_g)^0 \right] u \; dV_g\\
&\leq \max\{\frac{1}{a_m},1,\underline{a}_0\}\int_{M} u\left[\sum_{j=0}^m a_j(-\Delta_g)^j\right] u \; dV_g
= \max\{\frac{1}{a_m},1,\underline{a}_0\} \Vert u\Vert_{\textbf{a}}^2.
\end{align*}
On the other hand, by Lemmas \ref{Lemma:EquivalenceUsualNorms} and \ref{Lemma:EstimatesDerivatives}, there exists constants $C,C'>0$ such that
\begin{align*}
\Vert u\Vert_{\textbf{a}}^2 &= \int_{M} u\left[\sum_{j=0}^m a_j(-\Delta_g)^j\right] u \; dV_g
\leq \max\{\overline{a}_0,a_1,\ldots,a_m\}\int_M u \left[\sum_{j=0}^m (-\Delta_g)^j\right] u \;dV_g\\
&= \max\{\overline{a}_0,a_1,\ldots,a_m\}\Vert u\Vert_\Delta^2
\leq C \max\{\overline{a}_0,a_1,\ldots,a_m\}\Vert u\Vert_{H_g^m(M)}^2\\
&\leq C' C \max\{\overline{a}_0,a_1,\ldots,a_m\}\Vert u\Vert_{\ast}^2,
\end{align*}
and the result follows by density of $C^\infty(M)$ in $H_g^m(M)$.
\end{proof}

\section*{Acknowledgments} 

 The authors would like to thank Monica Clapp  for recommending the reference \cite{SzulkinWeth2010}.

\bibliographystyle{siam}
\bibliography{References}

\begin{flushleft}

\textbf{Romulo D. Carlos}\\

Departamento de Matemáticas, Facultad de Ciencias\\
Universidad Nacional Autónoma de México\\
Circuito Exterior, Ciudad Universitaria\\
04510 Coyoacán, Ciudad de México, Mexico\\
    \texttt{mathdiaz.unb@gmail.com}

\bigskip

\textbf{Juan Carlos Fernández}\\

Departamento de Matemáticas, Facultad de Ciencias\\
Universidad Nacional Autónoma de México\\
Circuito Exterior, Ciudad Universitaria\\
04510 Coyoacán, Ciudad de México, Mexico\\
\texttt{jcfmor@ciencias.unam.mx}

 \bigskip   
    
\textbf{María de los Ángeles Sandoval-Romero}

Departamento de Matemáticas, Facultad de Ciencias\\
Universidad Nacional Autónoma de México\\
Circuito Exterior, Ciudad Universitaria\\
04510 Coyoacán, Ciudad de México, Mexico\\
\texttt{selegna@ciencias.unam.mx}

\end{flushleft}

\end{document}